\documentclass[a4paper]{amsart}
\usepackage{amssymb}
\usepackage{amsmath}
\newcommand{\Span}[1]{\langle #1\rangle}

\newcommand{\R}{\mathbb{R}}
\newcommand{\triv}{\underline{\R}^{n+1,1}}
\newcommand{\Lor}{\R^{n+1,1}}
\newcommand{\half}{\tfrac{1}{2}}

\DeclareMathOperator{\End}{End} 
\DeclareMathOperator{\der}{d} \DeclareMathOperator{\Ort}{O}

\DeclareMathOperator{\sk}{o} \DeclareMathOperator{\Mob}{M\ddot{o}b}
\DeclareMathOperator{\Ad}{Ad} 
\newcommand{\D}{\mathrm{d}} 
 
\newcommand{\del}{\partial}
\newtheorem{theorem}{Theorem}[section]
\newtheorem{corollary}[theorem]{Corollary}

\newtheorem{proposition}[theorem]{Proposition}
\theoremstyle{definition}

\newtheorem{remark}[theorem]{Remark}
\newtheorem{definition}[theorem]{Definition}
\theoremstyle{plain}
\usepackage{ifpdf}
\ifpdf
\usepackage[pdftex,plainpages=false,pdfpagelabels,linktocpage,pdfstartview=FitH,colorlinks=true,linkcolor=blue,citecolor=magenta]{hyperref}
\else
\usepackage[hypertex,linkcolor=cyan]{hyperref}
\fi
\title{Special isothermic surfaces of type $d$}
\author{F.E. Burstall}
\address{Department of Mathematical Sciences\\ University of Bath\\
Bath BA2 7AY\\UK} \email{feb@maths.bath.ac.uk}
\author{S.D. Santos}
\address{Universidade de Lisboa\\ Faculdade de Ci\^encias\\ Departamento
de Matem\'atica\\ CMAF\\1749-016 Lisboa\\Portugal}
\email{susantos@ptmat.fc.ul.pt}

\begin{document}
\begin{abstract}
The special isothermic surfaces, discovered by Darboux in connection
with deformations of quadrics, admit a simple explanation via the
gauge-theoretic approach to isothermic surfaces.  We find that they
fit into a heirarchy of special classes of isothermic surface and
extend the theory to arbitrary codimension.
\end{abstract}
\maketitle
\section*{Introduction}
Isothermic surfaces, that is, surfaces which admit conformal
curvature line coordinates away from umbilics, were the focus of
intense activity by the geometers of the turn of the last century
\cite{Bia05,Bia05a,Cal03,Cal15,Dar99b}.  In recent times the subject
has attracted new interest initiated by
Cie\'sli\'nski--Goldstein--Sym \cite{CieGolSym95} who pointed out the
relationship with integrable systems.  Consequently, a highly
developed modern theory of isothermic surfaces has emerged which
\emph{inter alia} extends the classical theory to arbitrary
codimension and more exotic target spaces
\cite{BurHerPedPin97,Bur06,KamPedPin98,HerPed97,Her97,Sch01}.

However, a substantial part of the motivation and focus of the
classical geometers such as Bianchi and Darboux was on a particular
class of isothermic surfaces, the \emph{special isothermic surfaces},
which arise from the study of surfaces which are isometric to a
quadric.  These were isothermic surfaces in $\mathbb{R}^3$ or, more
generally, a $3$-dimensional space-form of section curvature $K$
characterised by the following rather unpromising equation on the
principal curvatures:
\begin{equation*}
e^{2\theta}(H_{u}^{2}+H_{v}^{2})+\frac{1}{4}M^{2}+AM-2BH+CL+D+\frac{1}{4}L^{2}K=0.
\end{equation*}
Here, the first and second fundamental forms of the surface are respectively
\begin{equation*}
I=e^{2\theta}(\D u^{2}+\D v^{2})  \mbox{ and }
 I\!I=e^{2\theta}(k_{1}\D u^{2}+k_{2}\D v^{2}),
\end{equation*}
$H=\frac{k_{1}+k_{2}}{2}$ is the mean curvature,
$L=e^{2\theta}(k_{1}-k_{2})$, $M=-HL$ while $A,B,C,D$ are real constants. 

The geometry from which this condition arises is as follows.  An
isothermic surface $F$ admits Darboux transforms depending on an
initial condition and a spectral parameter $m\in\mathbb{R}^{\times}$:
these are new isothermic surfaces which geometrically are the second
envelopes of certain sphere congruences enveloping $F$ and
analytically are solutions of a completely integrable system of
linear differential equations (a parallel section for a flat
connection).  Each such gives rise to a congruence of circles which
cut the surface and its transform orthogonally at corresponding
points and, through each such circle, there is, generically, a unique
plane in $\mathbb{R}^{3}$---this is the \emph{circle-plane of the
transform}.  Darboux \cite{Dar99b} proves that an isothermic surface
is special of class $(A,B,C,D)$ if and only if it admits (possibly
complex conjugate) Darboux transforms with distinct spectral
parameters for which the circle-planes coincide.  In this case, there
are three such transforms, the \emph{complementary surfaces}, which
can be constructed algebraically without integrations and whose
spectral parameters are the roots of the cubic equation
\begin{equation}\label{eq:14}
(m-A)^{2}m-Dm+2BC=0.
\end{equation}
Finally, the enveloping surface of this common congruence of
circle-planes is isometric to a quadric whose coefficients depend on
$(A,B,C,D)$ also.

Our purpose in this paper is to show that these special isothermic
surfaces have a simple explanation in terms of the integrable systems
approach to isothermic surfaces.  As fruits of this analysis, we
shall generalise the classical theory to arbitrary codimension and
see that the special isothermic surfaces are a particular case of a
hierarchy of natural classes of isothermic surfaces filtered by an
integer $d$.  In codimension $1$, the first three of these classes
are the maps to $S^{2}$ ($d=0$); the surfaces of constant mean
curvature ($d=1$) and the special isothermic surfaces described above
($d=2$).

Here is the basic idea: an isothermic surface is characterised by the
existence of a pencil of flat metric connections $\nabla^t$ on a
trivial $\mathbb{R}^{n+1,1}$-bundle.  In this formulation, Darboux
transforms with parameter $m$ are $\nabla^m$-parallel null line
subbundles.  Following an idea of the first author and Calderbank
\cite{BurCal}, we say that an isothermic surface is \emph{special of
type $d$} if there is a family $p(t)$ of $\nabla^t$-parallel sections
whose dependence on $t$ is polynomial of degree $d$.

It is already proved in \cite{BurCal} that special isothermic
surfaces of type $1$ are generalised $H$-spaces in a space-form and,
in section \ref{csc}, we prove that, modulo degeneracies, the special
isothermic surfaces of type $2$ in codimension $1$ are precisely the
special isothermic surfaces of Bianchi and Darboux.

In this setting, the complementary surfaces are easy to understand:
since the connections $\nabla^t$ are metric, the polynomial
$(p(t),p(t))$ has constant coefficients (in the classical case, this
is the polynomial \eqref{eq:14}) and the values $p(m)$, at roots $m$,
span the complementary surfaces.

Moreover, the behaviour of special isothermic surfaces under the
various transforms (Darboux, Christoffel, Calapso) of the theory is
straightforward to deduce.  In all cases, one has explicit gauge
transformations, with simple dependence on $t$, relating the pencils of
flat connections of the original surface and the transform and so can
readily construct new families of parallel sections.  In this way, we
efficiently extend results of Bianchi and Calapso to our entire
hierarchy of special isothermic surfaces in arbitrary codimension.
All this is the content of section \ref{Transformations of sis}.

Finally, in section \ref{spherical system and sphere-planes}, we
return to the roots of the subject and show that Darboux's
characterisation via circle-planes extends, \emph{mutatis mutandis},
to special isothermic surfaces of type $2$ in any space-form and
that, moreover, the special isothermic surfaces of type $2$ in
$\mathbb{R}^n$ give rise to surfaces isometric to quadrics.

\section{Preliminaries}\label{preliminaries}

\subsection{Isothermic and special isothermic surfaces}
\label{sec:isoth-spec-isoth}

\begin{definition}
Let $f:\Sigma\to E$ be an immersion of a surface
$\Sigma$ in an $n$-dimensional Riemannian manifold $E$, with
$n\geq3$. $f$ is called an \emph{isothermic surface} if there exist,
away from umbilic points, coordinates $(u,v)$ of $\Sigma$ such that
the first fundamental form $I:=(\D f,\D f)$ can be written as
\begin{equation*}
I=e^{2\theta}(\D u^{2}+\D v^{2}),
\end{equation*}
for some smooth real function $\theta$, and such that
$\frac{\partial}{\partial u}$ and $\frac{\partial}{\partial v}$
diagonalise simultaneously all shape operators.
\end{definition}

In this case $(u,v)$ are called \emph{conformal curvature line
coordinates} of $f$. More invariantly, a surface $f$ is isothermic if there
is a holomorphic quadratic differential $q$ which commutes with the
second fundamental form of $f$: away from zeros of $q$, one can find
a holomorphic coordinate $z$ for which $q^{2,0}=\D z^{2}$ and then
$z=u+iv$ provides the conformal
curvature line coordinates.

As we shall see, the class of isothermic surfaces in the
$n$-sphere $S^{n}$ is M\"{o}bius invariant, thus preserved by
conformal diffeomorphisms.

Quadrics, surfaces of revolution, cones and cylinders in
$\mathbb{R}^{3}$ are examples of isothermic surfaces as are constant
mean curvature surfaces in any $3$-dimensional space form.

The following definition is due to Darboux \cite{Dar99e}, and then
developed by Bianchi \cite{Bia05,Bia05a}.

\begin{definition}
\label{classical definition of special in space-forms}Let
$f:\Sigma\to E$ be an isothermic surface in a
$3$-dimensional space-form $E$ with (sectional) curvature
$K$. Consider conformal curvature line coordinates $(u,v)$ of $f$
with conformal factor $e^{2\theta}$ and principal curvatures $k_{1}$
and $k_{2}$ so that the first and second fundamental forms of $f$ are
given by
\begin{equation*}
I=e^{2\theta}(\D u^{2}+\D v^{2})  \text{ and }
 I\!I=e^{2\theta}(k_{1}\D u^{2}+k_{2}\D v^{2}).
\end{equation*}
$f$ is called a \emph{special isothermic surface} if there are real
constants $A$, $B$, $C$, and $D$ such that
\begin{equation}\label{Bianchi's equation generalized}
e^{2\theta}(H_{u}^{2}+H_{v}^{2})+\frac{1}{4}M^{2}+AM-2BH+CL+D+\frac{1}{4}L^{2}K=0,
\end{equation}
where $H=\frac{k_{1}+k_{2}}{2}$ is the mean curvature of $f$,
$L=e^{2\theta}(k_{1}-k_{2})$ and $M=-HL$.

In this case, we say that $f$ is a \emph{special isothermic surface
of class} $(A,B,C,D)$, with respect to the coordinates $(u,v)$.
\end{definition}

Equation \eqref{Bianchi's equation generalized} will be called the
\emph{Darboux-Bianchi condition}.

Constant mean curvature surfaces are examples of special isothermic
surfaces (in this case the class is not unique, for a given pair of
conformal curvature line coordinates).

\subsection{Conformal submanifold geometry}
\label{sec:conf-subm-geom}

\subsubsection{Conformal geometry of the sphere}
\label{sec:conf-geom-sphere}

We wish to study isothermic and, in particular, special isothermic
surfaces in the $n$-sphere from a conformally invariant view-point.
For this, we find a convenient setting in Darboux's light-cone model
of the conformal $n$-sphere.  So contemplate the light-cone
$\mathcal{L}$ in the Lorentzian vector space $\mathbb{R}^{n+1,1}$ and
its projectivisation $\mathbb{P}(\mathcal{L})$.  This last has a
conformal structure where representative metrics $g_{\sigma}$ arise from
never-zero sections $\sigma$ of the tautological bundle
$\pi:\mathcal{L}\to\mathbb{P}(\mathcal{L})$ via
\begin{equation*}
g_{\sigma}(X,Y)=(\D\sigma(X),\D\sigma(Y)).
\end{equation*}
Among these metrics are constant sectional curvature metrics given by
conic sections: indeed, for $w\in\mathbb{R}^{n+1,1}_{\times}$, set
$E(w):=\{v\in\mathcal{L}:(v,w)=-1\}$ to obtain an $n$-dimensional
submanifold of $\mathbb{R}^{n+1,1}$ with induced metric of constant
sectional curvature $-(w,w)$.  In fact, for $w$ non-null,
orthoprojection onto $\langle w\rangle^{\perp}$ induces an isometry
between $E(w)$ and $\{u\in\langle
w\rangle^{\perp}:(u,u)=-{1}/{(w,w)}\}$ while, when $w$ is null, for
any choice of $v_{0}\in E(w)$, orthoprojection onto $\langle
v_{0},w\rangle^{\perp}$ restricts to an isometry of $E(w)$.  By
construction, the bundle projection $\pi$ restricts to give a
conformal diffeomorphism between $E(w)$ and an open subset of
$\mathbb{P}(\mathcal{L})$.  In particular, choosing unit time-like
$w$ identifies $\mathbb{P}(\mathcal{L})$ with the conformal
$n$-sphere while, for $w$ null, the diffeomorphism
$\mathbb{P}(\mathcal{L})\setminus\{\langle w\rangle\}\cong \langle
v_0,w\rangle^{\perp}$ is essentially stereoprojection.

The M\"obius group $\Mob(n)$ of conformal diffeomorphisms of $S^n$ is
now readily identified, thanks to Liouville's Theorem, with an open
subgroup of the orthogonal group $\mathrm{O}(n+1,1)$ acting in the
obvious way on $\mathbb{P}(\mathcal{L})$.  This action is transitive
with parabolic stabilisers and each such stabiliser has abelian
nilradical so that $S^{n}$ is a \emph{symmetric $R$ space}, c.f.~\cite{BurDonPedPin09}.

This view-point gives the conformal geometry of the sphere a
projective linear flavour.  In particular, $k$-dimensional subspheres
of $S^{n}$ are identified with $(k+1,1)$-planes\footnote{Thus
$(k+2)$-dimensional linear subspaces of $\mathbb{R}^{n+1,1}$ with
induced inner product of signature $(k+1,1)$.} via
\begin{equation*}
V\mapsto\mathbb{P}(\mathcal{L}\cap V)\subset\mathbb{P}(\mathcal{L}).
\end{equation*}
Thus, the set of $k$-spheres in $S^{n}$ is identified with a
Grassmannian and so is a (pseudo-Riemannian) symmetric space for
$\mathrm{O}(n+1,1)$.

\subsubsection{Submanifolds and sphere congruences}
\label{sec:surf-sphere-congr}

A map $f:\Sigma\to S^{n}\cong\mathbb{P}(\mathcal{L})$ is the same
as a null line subbundle $\Lambda$ of the trivial bundle
$\underline{\mathbb{R}}^{n+1,1}=\Sigma\times\mathbb{R}^{n+1,1}$ via
\begin{equation*}
f(x)=\Lambda_x,
\end{equation*}
for all $x\in\Sigma$.  Given such a $\Lambda$, we define
\begin{equation*}
\Lambda^{(1)}:=\langle F,\D F(T\Sigma)\rangle,
\end{equation*}
where $F$ is any lift (never vanishing section) of $\Lambda$.  Then
$\Lambda$ is an immersion if and only if $\Lambda^{(1)}$ is a
subbundle of $\underline{\R}^{n+1,1}$ of rank $\dim\Sigma+1$.

\begin{definition}
A \emph{$k$-sphere congruence} is a map of $\Sigma$ into the space of
$k$-spheres or, equivalently, a subbundle $V$ of
$\underline{\R}^{n+1,1}$ with fibres of signature $(k+1,1)$.

A sphere congruence $V$ is \emph{enveloped by} $\Lambda$ if
$\Lambda^{(1)}\subset V$.

A sphere congruence $V$ is \emph{orthogonal} to $\Lambda$ if
$\Lambda\subset V$ and $\Lambda^{(1)}\subset \Lambda\oplus
V^{\perp}$.
\end{definition}

Here is the geometry: a sphere congruence $V$ is enveloped by
$\Lambda$ if each sphere $\mathbb{P}(\mathcal{L}\cap V_{x})$ is
tangent to the submanifold $\Lambda$ at $x$, for all $x\in\Sigma$.
Similarly, $V$ is orthogonal to $\Lambda$ if the corresponding
$k$-sphere meets the submanifold with orthogonal tangent planes at
each $x\in\Sigma$.

There are many sphere congruences enveloped by an immersion $\Lambda$
but a canonical choice is afforded by the \emph{central sphere
congruence} \cite{Ger31} defined as follows:
\begin{equation*}
V_{csc}:=\langle F,\D F(T\Sigma),\Delta
F\rangle=\Lambda^{(1)}\oplus\langle\Delta F\rangle,
\end{equation*}
where $F$ is a lift of $\Lambda$, and $\Delta F$ is the Laplacian of
$F:\Sigma\to\mathbb{R}^{n+1,1}$, with respect to the
metric $(\D F,\D F)$.  This is easily seen to be independent of
choices.

\subsection{Isothermic surfaces and their transformations}
\label{sec:isoth-surf-their}

\subsubsection{Gauge theoretic formulation}
\label{sec:invar-form}

There is a gauge-theoretic formulation of the isothermic surface
condition \cite{BurDonPedPin09} that will be basic in all that follows.  For
this, we begin by recalling the isomorphism
$\bigwedge^{2}\mathbb{R}^{n+1,1}\cong\sk(\mathbb{R}^{n+1,1})$ given by
\begin{equation*}
(u\wedge v)w=(u,w)v-(v,w)u,
\end{equation*}
for $u,v,w\in\Lor$.
A map $\Lambda$ gives rise to a subbundle of abelian subalgebras
$\Lambda\wedge\Lambda^{\perp}$ of the trivial Lie algebra bundle
$\sk(\underline\R^{n+1,1})$.
\begin{proposition}[{\cite[\S3.2.1]{BurDonPedPin09}\cite[Lemma~8.6.12]{Her03}}]
An immersion $\Lambda:\Sigma\to S^{n}\cong\mathbb{P}(\mathcal{L})$ is
isothermic if and only if there is a non-zero closed $1$-form
$\eta\in\Omega^{1}\otimes\sk(\Lor)$ taking values in
$\Lambda\wedge\Lambda^{\perp}$.
\end{proposition}
The classical and gauge-theoretic view-points are related as follows:
given $\eta\in\Omega^1(\Lambda\wedge\Lambda^{\perp})$, we define a
$2$-tensor $q$ on $\Sigma$ by
\begin{equation*}
\half q(X,Y)F=\eta_{X}\D F(Y),
\end{equation*}
for any lift $F$ of $\Lambda$.  This is well-defined and independent
of choices since $\eta\Lambda$ vanishes while $\eta\Lambda^{\perp}$
takes values in $\Lambda$.  Then $\eta$ is closed if and only if $q$
is a holomorphic quadratic differential which commutes with the
trace-free second fundamental form of $\Lambda$ \cite{BurCal}.
Moreover, in this case, $\eta$ takes values in
$\Lambda\wedge\Lambda^{(1)}$ and is uniquely determined by
$q$. Indeed, fix a lift $F$ and let $Q$ be the symmetric trace-free
endomorphism of $T\Sigma$ for which
\begin{equation*}
\half q(X,Y)= (\D F(QX),\D F(Y)).
\end{equation*}
Then
\begin{equation}
\label{eq:1}%Q-formulation of \eta
\eta=F\wedge\D F\circ Q.
\end{equation}
In particular, if $z=u+iv$ are conformal curvature line coordinates
with $q^{2,0}=\D z^{2}$ and $(\D F,\D F)=e^{2\theta}(\D u^{2}+\D
v^{2})$ then
\begin{equation}
\label{eq:2}%\eta via u and v
\eta=e^{-2\theta}F\wedge(-F_u\D u+F_v\D v).
\end{equation}
To summarise:
\begin{proposition}\label{Q symmetric and trace-free}
Let $(\Lambda,\eta)$ be an isothermic surface in $S^{n}$.
Then
\begin{enumerate}
\item $\eta\in\Omega^{1}(\Lambda\wedge \Lambda^{(1)})$;
\item For any lift $F$ of $\Lambda$, there is a symmetric, trace-free
$Q\in\Gamma(\End(T\Sigma))$ for which $\eta=F\wedge(\D F\circ Q)$;
\item $\eta$ vanishes on at most a discrete set.
\end{enumerate}
\end{proposition}

This formulation of the isothermic condition is manifestly
conformally invariant: for $T\in \Ort(\mathbb{R}^{n+1,1})$ and
$(\Lambda,\eta)$ isothermic, it is clear that $(T\Lambda,
\Ad(T)\eta)$ is isothermic also.

Like the holomorphic quadratic differential to which it is
equivalent, the $1$-form $\eta$ is defined up to a real constant
scale.  However, unless $\Lambda$ takes values in a $2$-sphere, this
is the only ambiguity: $q$ is fixed (up to scaling by a real
function) by the requirement that it commute with the trace-free
second fundamental form of $\Lambda$, unless the latter vanishes.
The holomorphicity of $q$ then forces that scale to be constant.
Thus:
\begin{proposition}[{\cite[Proposition~2.4]{Bur06}}]
Let $(\Lambda,\eta)$ be an isothermic surface in $S^{n}$. Then
$\eta$ is unique up to (non-zero) real scale if and only if
$\Lambda$ is not contained in any $2$-sphere.
\end{proposition}

A fundamental observation on which the rich transformation theory of
isothermic surfaces rests is that setting
$(\nabla^{t}:=\D+t\eta)_{t\in\mathbb{R}}$ yields a pencil of flat
metric connections on $\underline{\mathbb{R}}^{n+1,1}$.  Indeed,
since $\Lambda\wedge\Lambda^{\perp}$ is abelian, the curvature of
$\nabla^{t}$ is $t\D\eta$ so that, for a non-zero $1$-form
$\eta\in\Omega^{1}(\Lambda\wedge \Lambda^{\perp})$, $(\Lambda,\eta)$
is isothermic if and only if $(\nabla^{t})_{t\in\mathbb{R}}$ is a
family of flat connections.  We now briefly recall how these
connections provide transforms of isothermic surfaces.

\subsubsection{Darboux transforms}
\label{sec:darb-transf-an}

Let $(\Lambda,\eta)$ be an isothermic surface with pencil of flat
connections $\nabla^{t}$.  Darboux transforms of $\Lambda$ are
spanned by null $\nabla^m$-parallel sections of $\triv$.  More precisely:
\begin{definition}
An immersed surface $\hat{\Lambda}:\Sigma\to
\mathbb{P}(\mathcal{L})$ is a \emph{Darboux transform with parameter
$m$}, $m\in\R^{\times}$, of $\Lambda$ if
\begin{enumerate}
\item $\hat{\Lambda}\cap\Lambda=\{0\}$;
\item $\hat{\Lambda}$ is $(\D+m\eta)$-parallel.
\end{enumerate}
\end{definition}

The point here is that a Darboux transform $\hat\Lambda$ is also
isothermic and there is a simple explicit gauge transformation
relating the two pencils of flat connections.  For this, we define a
family of orthogonal gauge transformations
$\Gamma_{\Lambda}^{\hat{\Lambda}}(c)$, $c\in\R^{\times}$, of $\triv$ by
\begin{equation}\label{gauge transformation}
\Gamma_{\Lambda}^{\hat{\Lambda}}(c)=
\begin{cases}
c & \text{on $\hat{\Lambda}$,}\\
1& \text{on $(\Lambda\oplus\hat{\Lambda})^{\perp}$,}\\
c^{-1} & \text{on $\Lambda$.}
\end{cases}
\end{equation}
We now have:
\begin{proposition}[{\cite[Theorem~3.10]{BurDonPedPin09}}]\label{thm:darboux}
Let $\hat\Lambda$ be a Darboux transform with parameter $m$ of an
isothermic surface $(\Lambda,\eta)$.  Then $\hat\Lambda$ is
isothermic and the $1$-form
$\hat\eta\in\Omega^1(\hat\Lambda\wedge\hat\Lambda^{\perp})$ can be
chosen so that:
\begin{enumerate}
\item
$\Gamma_{\Lambda}^{\hat{\Lambda}}(1-\frac{t}{m})\cdot(\D+t\eta)=\D+t\hat\eta$,
for all $t\neq m$;
\item $(\Lambda,\eta)$ is a Darboux transform with parameter $m$ of $(\hat\Lambda,\hat\eta)$.
\end{enumerate}
\end{proposition}
Here the action of the gauge transformation on connections is the
usual left action: $\Gamma_{\Lambda}^{\hat{\Lambda}}(1-\frac{t}{m})\cdot(\D+t\eta)
=\Gamma_{\Lambda}^{\hat{\Lambda}}(1-\frac{t}{m})
\circ(\D+t\eta)\circ\Gamma_{\Lambda}^{\hat{\Lambda}}(1-\frac{t}{m})^{-1}$.
\begin{remark}
The prescription of $\hat\eta$ in Proposition~\ref{thm:darboux}
amounts to the demand that the two holomorphic quadratic
differentials coincide: $q=\hat q$.
\end{remark}

\subsubsection{Christoffel transforms}\label{Christoffel
transforms}

Fix a pair $v_{\infty},v_{0}\in\mathcal{L}$ with
$(v_{0},v_{\infty})=-1$ so that $v_0\in E(v_{\infty})$ and
conversely.  Set $\R^{n}=\langle v_0,v_{\infty}\rangle^{\perp}$.  Orthoprojection is an isometry
$E(v_{\infty})\to\R^{n}$ with inverse given by
\begin{equation*}
x\mapsto\exp(x\wedge v_\infty)v_{0}=v_0+x+\half(x,x)v_{\infty}.
\end{equation*}

Now let $(\Lambda,\eta)$ be isothermic with lift $F=\exp(f\wedge
v_{\infty})v_{0}:\Sigma\to E(v_{\infty})$.  Then we can write
\begin{equation*}
\eta=\Ad(\exp(f\wedge v_{\infty}))\,\omega,
\end{equation*}
for $\omega\in\Omega^1(\Span{v_0}\wedge\Span{v_0}^{\perp})$.  It
follows from $\D\eta=0$ that $\D\omega=0$
\cite[Proposition~3.3]{BurDonPedPin09} so that locally we have $\omega=\D
f^{c}\wedge v_{0}$, for a map $f^{c}:\Sigma\to\R^n$ defined up to a
translation.

We now swop the roles of $f$ and $f^{c}$, $v_{0}$ and $v_{\infty}$ to
make:
\begin{definition}
The \emph{Christoffel transforms} of $(\Lambda,\eta)$, with respect
to the pair $(v_{\infty},v_{0})$, are the surfaces
$(\Lambda^{c}:=\langle
F^{c}\rangle,\eta^c)$, where
\begin{align*}
F^{c}&=\exp(f^c\wedge v_{0})v_{\infty}:\Sigma\to E(v_0)\\
\eta^c&=\Ad(\exp(f^{c}\wedge v_0))\,(\D f\wedge v_{\infty}).
\end{align*}
\end{definition}
We have:
\begin{proposition}[{\cite[Lemma~3.13]{BurDonPedPin09}}]
Let $(\Lambda^{c},\eta^{c})$ be a Christoffel transform of
$(\Lambda,\eta)$.  Then
\begin{enumerate}
\item $(\Lambda^{c},\eta^{c})$ is an isothermic surface;
\item For all $t\in\R^{\times}$,
$\Gamma^{c}(t)\cdot(\D+t\eta)=\D+t\eta^{c}$, where the gauge
transformation $\Gamma^{c}(t)$ is given by
\begin{equation*}
\Gamma^{c}(t)=\exp(f^{c}\wedge v_{0})\circ \Gamma_{\langle
v_{0}\rangle}^{\langle v_{\infty}\rangle}(t)\circ \exp(-f\wedge
v_{\infty});
\end{equation*}
\item $(\Lambda,\eta)$ is a Christoffel transform of $(\Lambda^{c},\eta^c)$.
\end{enumerate}
\end{proposition}

The geometry behind all this is that the stereoprojection $f^{c}$ of
a Christoffel transform has parallel tangent planes to those of $f$
and induces the same conformal structure on $\Sigma$ as $f$ while
inducing the opposite orientation on $T\Sigma$ \cite{Chr67,Pal88}.

\subsubsection{T-transforms}

Let $(\Lambda,\eta)$ be an isothermic surface in $S^{n}$. Since each
$\nabla^{s}=\D+s\eta$ is a flat, metric connection on $\triv$, there
are local orthogonal gauge transformations $\Phi_{s}$, for each
$s\in\R$, with $\Phi_s\cdot\nabla^{s}=\D$.  With this understood, we have:
\begin{definition}The \emph{$T$-transforms} of $(\Lambda,\eta)$ are the surfaces
$(\Lambda_{s},\eta_{s})$, $s\in\mathbb{R}$, where
$\Lambda_{s}=\Phi_s\Lambda$ and $\eta_{s}=\Ad(\Phi_{s})\eta$.
\end{definition}
Note that each $\Phi_{s}$, and so each $T$-transform $\Lambda_s$, is
defined up to the action of the M\"obius group.

We have:
\begin{proposition}[{\cite[\S3.3]{BurDonPedPin09}}]
Let $(\Lambda_{s},\eta_{s})$ be a $T$-transform of
$(\Lambda,\eta)$, $s\in\R$.  Then
\begin{enumerate}
\item $(\Lambda_s,\eta_s)$ is an isothermic surface;
\item For all $t\in\R$,
$\Phi_{s}\cdot(\D+(t+s)\eta)=\D+t\eta_{s}$.
\end{enumerate}
\end{proposition}

This spectral deformation of isothermic surfaces coincides with that
introduced, independently, by Calapso \cite{Cal03} and Bianchi
\cite{Bia05a} for the case $n=3$ and extended to higher codimension by
Burstall \cite{Bur06} and Schief \cite{Sch01}.

\section{Special isothermic surfaces of type $d$}

We now come to the central idea of this paper: consider an isothermic
surface $(\Lambda,\eta)$ with its pencil of flat connections
$\nabla^t$.  The theory of ordinary differential equations ensure
that we may find $\nabla^{t}$-parallel sections depending smoothly on
the spectral parameter $t$.  However, the existence of such sections
with \emph{polynomial} dependence on $t$ is a condition on $\Lambda$
of geometric significance.  Following Burstall--Calderbank
\cite{BurCal}, we are therefore led to make the following definition:

\begin{definition}Let $(\Lambda,\eta)$ be an isothermic surface in $S^{n}$ and
let $p(t)\in\Gamma(\underline{\mathbb{R}}^{n+1,1})[t]$. The
polynomial $p(t)$ is called a \emph{polynomial conserved quantity}
of $(\Lambda,\eta)$ if $p(t)$ is non-zero and
$\nabla^{t}p(t)\equiv0$.
\end{definition}

\begin{proposition}\label{csc}Let $(\Lambda,\eta)$ be an isothermic surface in $S^{n}$. If
$p(t)=\sum_{k=0}^d p_{k}t^{k}$ is a polynomial conserved quantity of
$(\Lambda,\eta)$ with degree $d\in\mathbb{N}_{0}$, then
\begin{enumerate}
\item $p_{0}$ is constant;
\item $p_{d}$ is a parallel section of $V_{csc}^{\perp}$, with respect
to the normal connection;
\item the polynomial $(p(t),p(t))\in\mathbb{R}[t]$, that is, $(p(t),p(t))$
has constant coefficients.
\end{enumerate}
\end{proposition}
\begin{proof}
Write $p(t)=\sum_{k=0}^d p_{k}t^{k}$ with $p_d$ non-zero.
Evaluating $(\D+t\eta)p(t)$ at $t=0$ we obtain $\D p_{0}=0$.  The top
two coefficients of $\nabla^tp(t)\equiv0$ read
\begin{subequations}
\begin{align}
0&=\eta p_d\label{eq:3}\\
0&=\D p_d+\eta p_{d-1},\label{eq:4}
\end{align}
\end{subequations}
where we have set $p_{-1}=0$. Fix a lift $F$ of
$\Lambda$. Proposition \ref{Q symmetric and trace-free} tells us that
we may write $\eta=F\wedge(\D F\circ Q)$ with
$Q\in\Gamma(\End(T\Sigma))$ symmetric, trace-free and bijective off a
discrete subset of $\Sigma$.  Now \eqref{eq:3} reads
\begin{equation*}
(F,p_{d})\D F\circ Q-(\D F\circ Q,p_{d})F=0,
\end{equation*}
and, since $F$ immerses, we see that $(p_{d},F)=0$ and $(\D F\circ
Q,p_{d})=0$ so that $(\D F,p_{d})=0$) off the zero-set of $Q$. Hence
$p_{d}\in\Gamma{\Lambda^{(1)}}^{\perp}$.

To see that $p_d$ takes values in $V_{csc}$, it remains to show that
$(\D*\D F,p_d)=0$, where $*$ is the Hodge star operator on $\Sigma$.
Now \eqref{eq:4} reads:
\begin{equation}\label{eq:5}
\D p_d+(F,p_{d-1})\D F\circ Q - (\D F\circ Q,p_{d-1})F=0,
\end{equation}
so that
\begin{align*}
(\D*\D F, p_d) &= \D(*\D F, p_d)-(*\D F\wedge \D p_d)\\
&=-(F,p_{d-1})(*\D F\wedge\D F\circ Q)
\end{align*}
which last vanishes since $Q$ is trace-free.
Thus, $p_{d}\in\Gamma V_{csc}^{\perp}$.  Moreover, from
\eqref{eq:5}, we see that $\D
p_{d}\in\Omega^{1}(\Lambda^{(1)})\subseteq\Omega^{1}(V_{csc})$,
whence $\nabla^{\perp}p_{d}=0$, since the normal connection
$\nabla^{\perp}$ is just the $V_{csc}^{\perp}$-component of $\D$.

Finally, since each $\nabla^{t}=\D+t\eta$ is a metric connection, and
$\nabla^{t}p(t)\equiv0$, we get
$\der(p(t),p(t))=2(\nabla^{t}p(t),p(t))\equiv0$.
\end{proof}
In particular, $p_{d}$ is always space-like while, $p_0$, if it is
non-zero, defines a space-form $E(p_0)$.

Observe that if $n=3$, all top terms of polynomial conserved
quantities are constant multiples of each other, so that whenever
there are two linearly independent polynomial conserved quantities of
degree $d$, there is necessarily a polynomial conserved quantity of
degree $d-1$.

\begin{definition}\label{definition}
An isothermic surface $(\Lambda,\eta)$ in $S^n$ is a \emph{special
isothermic surface of type $d$} (with respect to $p(t)$) if it admits
a polynomial conserved quantity $p(t)$ of degree $d$.

Moreover, for $w\in\Lor_{\times}$, we say that $(\Lambda,\eta)$ is
\emph{special isothermic of type $d$ in the space-form $E(w)$} if
$p_{0}\in\Span{w}$.
\end{definition}

The cases $d=0$ and $d=1$ have simple interpretations to which we now
turn.  Recall that a submanifold of $S^n$ is \emph{full} if it does
not lie in any proper sub-sphere.  We have:

\begin{proposition}An isothermic surface
$(\Lambda,\eta)$ in $S^{n}$ is a special isothermic surface of type $0$ if and
only if $\Lambda$ is not full.
\end{proposition}
\begin{proof}
Suppose that $(\Lambda,\eta)$ is a special isothermic surface
of type $0$, with respect to a polynomial $p(t)=p_0$. By
Proposition~\ref{csc}, $p_{0}$ is a constant in
$V_{csc}^{\perp}$. Set $W=\Span{p_0}^{\perp}$ so that $W$ is an
$(n,1)$-plane with $\Lambda\subset V_{csc}\subset W$ and $\Lambda$
has image in the $(n-1)$-sphere $\mathbb{P}(\mathcal{L}\cap W)$.

For the converse, if $\Lambda$ has image in the subsphere
$\mathbb{P}(\mathcal{L}\cap W)$, for some $(n,1)$-plane
$W=\Span{p_{0}}^{\perp}$, it is easy to see that $p(t):=p_{0}$ is a
polynomial conserved quantity of $(\Lambda,\eta)$.
\end{proof}

Let us now contemplate the case $d=1$.  Recall that a surface $F$ in a
space-form is a \emph{generalised $H$-surface} if it admits a
parallel unit normal vector field which has constant inner product
with the mean curvature vector of $F$ (see \cite[\S2.1]{Bur06}).  In
codimension $1$, a generalised $H$-surface is simply a surface of
constant mean curvature.  We now have:
\begin{proposition}[\cite{BurCal}]\label{H-surface}Let $(\Lambda,\eta)$ be a full
isothermic surface in $S^{n}$, with $n\geq3$.
Then $(\Lambda,\eta)$ is a special isothermic
surface of type $1$ in the space form $E(w)$, $w\in\Lor_{\times}$, if
and only if the lift
$F:\Sigma\to E(w)$ of $\Lambda$ is a generalised
$H$-surface.
\end{proposition}
\begin{proof}[Sketch proof]
Let $p(t)=p_{0}+p_{1}t$ be a polynomial conserved quantity of
$(\Lambda,\eta)$, with $p_{0}\in\langle w\rangle$ non-zero since
$\Lambda$ is full.  Without loss of generality, take $p_{1}$ to be of
unit length.  We can then write $p_{1}=(\mathbf{H},N)F+N$, for
$\mathbf{H}$ the mean curvature vector of $F$ and $N$ a unit parallel
section of $V_{w}^{\perp}$, thus a parallel unit normal to $F$.
Since $(p_{0},p_{1})$ is constant, it immediately follows that
$(\mathbf{H},N)$ is constant and $F$ is a generalised $H$-surface.

For the converse, if $F:\Sigma\to E(w)$ is a generalised $H$-surface
and $N$ the corresponding parallel unit normal, set $p(t):=w+p_{1}t$,
where $p_{1}:=(\mathbf{H},N)F+N$, and $\eta=F\wedge \D p_{1}$.  One
then shows that $\eta$ is closed and $(\D+t\eta)p(t)$ vanishes
identically.  For more details, see \cite{BurCal}.
\end{proof}

Note that in codimension $1$, the condition that an isothermic
surface be special of type $d=0,1$ is a differential equation on the
principal curvatures of order $d$.  This situation persists for all
$d$ \cite{San08} and, in particular, we now see that the case $d=2$
is very closely related to Darboux's notion formulated in
Definition~\ref{classical definition of special in space-forms}:

\begin{theorem}\label{theorem between pcq and new Bianchi's condition}
Let $(\Lambda,\eta)$ be an isothermic surface in $S^{3}$. Fix
$w\in\mathbb{R}^{4,1}_{\times}$ and consider the lift $F:\Sigma\to E(w)$
of $\Lambda$.  Let $(u,v)$ be conformal
curvature line coordinates corresponding to $\eta$ so that the first and
second fundamental forms of $F$ are given by
\begin{equation*}
I=e^{2\theta}(\D u^{2}+\D v^{2})\qquad
I\!I=e^{2\theta}(k_{1}\D u^{2}+k_{2}\D v^{2}).
\end{equation*}
Let $H=\frac{k_{1}+k_{2}}{2}$ be the mean curvature of $F$ and set
$L=e^{2\theta}(k_{1}-k_{2})$ and $M=-HL$.

Then $(\Lambda,\eta)$ is special isothermic of degree $2$ in $E(w)$
if and only if there are real constants $A$, $B$ and $C$ such that
\begin{equation}\label{new Bianchi condition}
\begin{cases}
H_{uu}+\theta_{u}H_{u}-\theta_{v}H_{v}-\frac{1}{2}Mk_{1}-Ak_{1}-Be^{-2\theta}+C-\frac{1}{2}L(w,w)=0\\
H_{vv}-\theta_{u}H_{u}+\theta_{v}H_{v}+\frac{1}{2}Mk_{2}+Ak_{2}-Be^{-2\theta}-C+\frac{1}{2}L(w,w)=0.
\end{cases}
\end{equation}
\end{theorem}
\begin{proof}
Set $W_{1}:=e^{-\theta}F_{u}$, $W_{2}:=e^{-\theta}F_{v}$ and let $N$
be a unit normal to $F$ in $E(w)$.  From \eqref{eq:2}, we have
\begin{equation*}
\eta=e^{-2\theta}F\wedge(-F_{u}\D u+F_{v}\D
v)=e^{-\theta}F\wedge(-W_1\D u+W_2\D v).
\end{equation*}

Now suppose that $(\Lambda,\eta)$ admits a polynomial conserved quantity
$p(t)=p_{0}+p_{1}t+p_{2}t^{2}$ with $p_{0}\in\langle w\rangle$.
By Proposition~\ref{csc}, we may scale $p$ by a constant to
ensure that $p_{2}=HF+N$. Denote by $B$
the constant such that $p_{0}=Bw$ and write
\begin{equation*}
p_{1}=\alpha F+\beta W_{1}+\gamma W_{2}+\delta N+\epsilon w,
\end{equation*}
for functions $\alpha$, $\beta$, $\gamma$, $\delta$ and $\epsilon$.
With $p_2=HF+N$, $\D p_{2}+\eta p_{1}=0$ is equivalent to
\begin{equation*}
\beta=-e^{\theta}H_{u},\quad \gamma=e^{\theta}H_{v}\quad\epsilon=\frac{1}{2}L.
\end{equation*}
Now compute the components of $\D p_{1}+\eta p_{0}$ along the frame
$F,W_1,W_2,N,w$ to conclude that the vanishing of this last is
equivalent to
\begin{subequations}
\begin{gather}
\label{eq:6}
\D\alpha=e^{2\theta}H_u(w,w)\D u+e^{2\theta}H_v(w,w)\D v\\
\label{eq:7}
H_{uv}+\theta_uH_v+\theta_vH_u=0\\
\label{eq:8}
e^{2\theta}\D H=\tfrac{1}{2}(L_u\D u-L_v\D v)\\
\label{eq:9}
\D\delta=e^{2\theta}\kappa_1H_u\D u-e^{2\theta}\kappa_2 H_v\D v\\
\label{eq:10}
\begin{split}
H_{uu}+\theta_{u}H_{u}-\theta_{v}H_{v}+k_{1}\delta-\alpha-Be^{-2\theta}&=0\\
H_{vv}-\theta_{u}H_{u}+\theta_{v}H_{v}-k_{2}\delta+\alpha-Be^{-2\theta}&=0.
\end{split}
\end{gather}
\end{subequations}
Of these, \eqref{eq:7} and \eqref{eq:8} are consequences of the
Codazzi equations and, using \eqref{eq:8}, \eqref{eq:6} and
\eqref{eq:9} can be written
\begin{equation*}
\D\alpha=\half(w,w)\D L,\qquad\D\delta=-\half\D M.
\end{equation*}
We therefore introduce constants $A$ and $C$ so that
$\delta=-(\frac{1}{2}M+A)$ and $\alpha=\frac{1}{2}L(w,w)-C$ and
conclude that $p(t)$ is a polynomial conserved quantity if and only
if we have
\begin{subequations}
\label{eq:11}
\begin{align}
p_{0}&=Bw,\\
p_{1}&=(\frac{1}{2}L(w,w)-C)F-e^{\theta}H_{u}W_{1}+e^{\theta}H_{v}W_{2}-(\frac{1}{2}M+A)N+\frac{1}{2}Lw\\
p_{2}&=HF+N
\end{align}
\end{subequations}
and \eqref{eq:10} holds.  These equations, which are the components
of $\D p_{1}+\eta p_0=0$ along $W_1\D u$ and $W_2\D v$ respectively,
are precisely \eqref{new Bianchi condition}.
\end{proof}

Let us relate this to the Bianchi--Darboux condition \eqref{Bianchi's
equation generalized} that characterises the classical special
isothermic surfaces.  First observe that, with $p_0,p_1,p_2$ defined
by \eqref{eq:11}, equation \eqref{Bianchi's equation generalized}
reads
\begin{equation*}
(p_1,p_1)+2(p_0,p_2)=D-A^2.
\end{equation*}
Thus a special isothermic surface of type $2$ satisfies the
Bianchi--Darboux condition, the coefficients of $(p(t),p(t))$ being
constant.  Moreover, the converse is almost true as well: if $\Lambda$ is an
isothermic surface satisfying \eqref{Bianchi's equation generalized},
we define $p(t)$ by \eqref{eq:11} and then see that
\begin{equation*}
(p(t),p(t))=t^4-2At^3+(D-A^2)t^2+2BCt+B^2(w,w)
\end{equation*}
is constant.  Now
\begin{equation*}
\D p(t)+t\eta p(t)=\omega_1W_1\D u+\omega_2W_2\D v
\end{equation*}
with $\omega_1,\omega_2$ determined by \eqref{eq:10} and we have
$0=\D(p(t),p(t))=2(\D p(t)+t\eta p(t),p(t))$, the $t^2$-coefficient
of which is
\begin{equation*}
0=(\omega_1W_1,p_1)\D u+(\omega_2W_2,p_1)\D
v=e^{\theta}(-H_{u}\omega_1\D u+H_v\omega_2\D v).
\end{equation*}
We conclude that when $H_{u}H_{v}$ is never zero\footnote{This
condition is quite strong, excluding, for example, cylinders and
surfaces of revolution.}, a surface is special isothermic of type
$2$ in a space-form $E(w)$ if and only if the Bianchi--Darboux
condition \eqref{Bianchi's equation generalized} holds and so is
special isothermic in the classical sense.

The class of special isothermic surface of type $d$ in $S^{n}$ is M\"{o}bius
invariant: if $\Lambda$ has polynomial conserved quantity $p(t)$ and
$T\in\Ort(\Lor)$ then $T\Lambda$ has polynomial conserved quantity
$Tp(t)$.  However, the property of being special isothermic in a
fixed space-form $E(w)$ is, of course, not M\"obius invariant:
$T\Lambda$ is special isothermic in $E(Tw)$ and not, as a rule, in
$E(w)$.

There is an exception to this last rule: if $p(0)=0$ then $\Lambda$
is isothermic of type $d$ in \emph{any} space-form $E(w)$.  This is
the case if and only if $q(t)=p(t)/t$ is a polynomial conserved
quantity of degree $d-1$.  In particular, taking $n=3$ and $d=2$, we
recover the discussion of Bianchi \cite[\S\S21--22]{Bia05} to conclude
that the special isothermic surfaces of type $2$ with $B=0$ are
precisely the images under the M\"obius group of constant mean
curvature surfaces in some space-form.

\section{Transforms of special isothermic surfaces}
\label{Transformations of sis}

We now show that the class of special isothermic surfaces is very
well-behaved with respect to the transformation theory rehearsed in
Section~\ref{sec:isoth-surf-their}.

\subsection{Darboux transforms of a special isothermic
surface} \label{section of Darboux transforms of a special
isothermic surface}

Let $(\Lambda,\eta)$ be an isothermic surface with Darboux transform
$(\hat\Lambda,\hat\eta)$.  From Proposition~\ref{thm:darboux}, we
have
\begin{equation*}
\Gamma_{\Lambda}^{\hat{\Lambda}}(1-{t}/{m})\cdot(\D+t\eta)=\D+t\hat\eta.
\end{equation*}
Thus, if $p(t)$ is a polynomial conserved quantity for
$(\Lambda,\eta)$,
$\Gamma_{\Lambda}^{\hat{\Lambda}}(1-\frac{t}{m})p(t)$ is
$\D+t\hat\eta$-parallel and rational in $t$ with at worst a simple
pole at $m$.

\begin{theorem}\label{general case}Let $(\Lambda,\eta)$ be a special isothermic surface
in $S^{n}$ of type $d\in\mathbb{N}_{0}$ with respect to a polynomial
$p(t)$. Let $\hat{\Lambda}$ be a Darboux transform of
$(\Lambda,\eta)$, with spectral parameter $m$, and let $G$ be a
$(\D+m\eta)$-parallel lift of $\hat{\Lambda}$. Then:

1) $(p(m),G)$ is constant;

2) if $(p(m),G)=0$, then $(\hat{\Lambda},\hat{\eta})$ is a special
isothermic surface of type $d$, with respect  to a polynomial
$\hat{p}(t)$ for which $\hat{p}(0)=p(0)$ and
$(\hat{p}(t),\hat{p}(t))=(p(t),p(t))$.
\begin{proof}
$(p(m),G)$ constant is an immediate consequence of $p(m)$ and $G$
being parallel sections with respect to the metric connection
$\D+m\eta$. Assume now that $p(m)$ is orthogonal to $\hat{\Lambda}$.
Since $m$ is a root of the polynomial $p(t)_{\Lambda}$ \footnote{For
each section $\varphi$ of $\underline{\mathbb{R}}^{n+1,1}$, write
\begin{equation*}
\varphi=\varphi_{\Lambda}+\varphi_{\hat{\Lambda}}+\varphi_{W^{\perp}}
\end{equation*}
for the decomposition of $\varphi$ corresponding to
$\underline{\mathbb{R}}^{n+1,1}=\Lambda\oplus\hat{\Lambda}\oplus
W^{\perp}$, where $W:=\Lambda\oplus\hat{\Lambda}$.}, we can define
the polynomial $\hat{p}(t)$ by
$\hat{p}(t)=\Gamma_{\Lambda}^{\hat{\Lambda}}(1-\frac{t}{m})p(t)$,
for all $t\in\mathbb{R}\backslash\{m\}$, that is, the polynomial
$\hat{p}(t)$ such that
\begin{equation*}
p(t)_{\Lambda}=(1-\frac{t}{m})\hat{p}(t)_{\Lambda},\;\hat{p}(t)_{\hat{\Lambda}}=(1-\frac{t}{m})p(t)_{\hat{\Lambda}}
\mbox{ and }\hat{p}(t)_{W^{\perp}}=p(t)_{W^{\perp}},
\end{equation*}
which satisfies all the conditions described on point 2.

\end{proof}
\end{theorem}

In particular, from the previous Theorem, we get a sufficient
condition for a Darboux transform of a (classical) special
isothermic surface to be again a special isothermic surface of the
same class (in a certain $3$-dimensional space-form). This classical
result can be seen in \cite{Bia05,Bia05a}. This Theorem encodes also a
sufficient condition for a Darboux transform of a constant mean
curvature surface in a certain space-form to be again a constant
mean curvature surface in the same space-form, with the same mean
curvature (with the obvious changes, we have the analogous result
for generalised H-surfaces).

\medskip

The next Theorem establishes two results: extends the classical
result due to Calapso \cite{Cal15} which claims that in
$\mathbb{R}^{3}$ all Darboux transforms of a constant mean curvature
surface are special isothermic surfaces; guarantees that the
sufficient condition of Theorem \ref{general case} is also a
necessary condition, with the assumptions that codimension is $1$
and that $(\Lambda,\eta)$ is not a special isothermic surface of
type less than $d$ (see \cite{Bia05} for the classical result).

\begin{theorem}\label{the reverse} Let $(\Lambda,\eta)$ be a special isothermic surface in $S^{n}$ of type
$d\in\mathbb{N}_{0}$ with respect to a polynomial $p(t)$, and let
$\hat{\Lambda}$ be a Darboux transform of $(\Lambda,\eta)$, with
parameter $m$. Then:

1) $(\hat{\Lambda},\hat{\eta})$ is a special isothermic surface of
type $d+1$, with respect to a polynomial $\hat{p}(t)$ for which
$\hat{p}(0)=p(0)$;

2) if $n=3$, $d\in\mathbb{N}$, $(\Lambda,\eta)$ is not a special
isothermic surface of type $d-1$, and $(\hat{\Lambda},\hat{\eta})$
is a special isothermic surface of type $d$, with respect to a
polynomial $\hat{p}(t)$ such that $\hat{p}(0)=p(0)$, then
$p(m)\in\Gamma(\hat{\Lambda}^{\perp})$. \footnote{This result is
obviously true if we take $d=0$, with arbitrary codimension.}
\begin{proof}Consider the $(d+1)$-th
degree polynomial
\begin{equation*}
q(t)=\big(1-\frac{t}{m}\big)p(t),
\end{equation*}
which is a polynomial conserved quantity of $(\Lambda,\eta)$. Since
$q(m)=0$, and therefore $q(m)\in\Gamma(\hat{\Lambda}^{\perp})$, we
guarantee that $(\hat{\Lambda},\hat{\eta})$ is a special isothermic
surface of type $d+1$, with respect to a polynomial $\hat{q}(t)$
such that $\hat{q}(0)=q(0)$, and then $\hat{q}(0)=p(0)$.

Assume now that  $n=3$, $d\in\mathbb{N}$ and that $(\Lambda,\eta)$
is not a special isothermic surface of type $d-1$. Suppose that
$(\hat{\Lambda},\hat{\eta})$ is a special isothermic surface of type
$d$, with respect to a polynomial $\hat{p}(t)$, such that
$\hat{p}(0)=p(0)$. Based on the first part of this proof, it follows
that the polynomial $q(t)$ satisfying
$q(t)=\Gamma_{\hat{\Lambda}}^{\Lambda}(1-\frac{t}{m})(1-\frac{t}{m})\hat{p}(t)$
is a polynomial conserved quantity of $(\Lambda,\eta)$, with degree
$d+1$.

Take the polynomial conserved quantity $p'(t):=(1-\frac{t}{m})p(t)$
of $(\Lambda,\eta)$. Since $q(0)=p'(0)$, consider the polynomial
$s(t)$ such that $ts(t)=q(t)-p'(t)$. One can prove that $q(t)-p'(t)$
has degree $d+1$ or $q(t)-p'(t)$ is the zero polynomial.

If $q(t)-p'(t)$ has degree $d+1$, we get that $s(t)$ is a polynomial
conserved quantity of $(\Lambda,\eta)$ with degree $d$, and then,
taking into account that the codimension is $1$, we conclude that
there is a real constant $\alpha$ such that $p(t)=\alpha s(t)$.
Consequently $p(m)_{\Lambda}=0$; if $q(t)-p'(t)\equiv0$, we get
$q(m)=p'(m)=0$. Take the polynomial $a(t)$ of degree $d$ such that
$q(t)=(1-\frac{t}{m})a(t)$. In this case we also have $p(t)=\alpha
a(t)$, for some real $\alpha$, and $a(m)_{\Lambda}=0$ (because
$(1-\frac{1}{m}t)\hat{p}(t)_{\Lambda}=a(t)_{\Lambda}$).
\end{proof}
\end{theorem}

From this Theorem we obtain in particular examples of special
isothermic surfaces of type $2$, taking a special isothermic surface
of type $1$, i.e., a generalised $H$-surface in some space-form
(which does not live in a $2$-sphere) and considering its Darboux
transforms.

\medskip

For a given special isothermic surface with respect to a polynomial
$p(t)$, there are particular Darboux transforms which are obtained
without integration. They arise from the roots of $(p(t),p(t))$.

\begin{definition}\label{definition of complementary surfaces}Let $(\Lambda,\eta)$ be a special isothermic surface in $S^{n}$, with respect
to $p(t)$. The Darboux transforms $\hat{\Lambda}$ of
$(\Lambda,\eta)$ such that $\hat{\Lambda}=\langle p(m)\rangle$, for
some $m\in\mathbb{R}^{\times}$ satisfying $(p(m),p(m))=0$, are
called the \emph{complementary surfaces} of $(\Lambda,\eta)$, with
respect to $p(t)$.
\end{definition}

We obtain therefore at most $2d$ complementary surfaces for each
special isothermic surface of type $d$.

Note that for $d=1$ with $0\neq p(0)\in\mathcal{L}$, we have exactly
one complementary surface. Fixing a $v_{0}\in E(p(0))$, it is
straightforward to check that it is the only surface in $S^{n}$ such
that its projection in $\mathbb{R}^{n}=\langle
v_{0},p(0)\rangle^{\perp}\cong E(p(0))$ is simultaneously a Darboux
transform and a Christoffel transform of the projection $f$ of
$\Lambda$ in $\mathbb{R}^{n}$. In $\mathbb{R}^{n}$ it is given by
$f+\frac{1}{H}N$, where $N$ is a unit normal section of $f$,
parallel with respect to the normal connection such that
$H=(\mathbf{H},N)$ is constant ($\mathbf{H}$ is denoting the mean
curvature vector of $f$). We get in particular that this
complementary surface of $\Lambda$ is also a generalised H-surface
in the space-form $E(p(0))$. This result is a particular case of the
following Corollary of Theorem \ref{general case}.

\begin{corollary} Let $(\Lambda,\eta)$ be a special isothermic surface in $S^{n}$
of type $d\in\mathbb{N}$ with respect to a polynomial $p(t)$. Let
$\hat{\Lambda}$ be a complementary surface of $(\Lambda,\eta)$ with
respect to $p(t)$. Then $(\hat{\Lambda},\hat{\eta})$ is a special
isothermic surface of type $d$, with respect to a polynomial
$\hat{p}(t)$, such that $\hat{p}(0)=p(0)$ and
$(\hat{p}(t),\hat{p}(t))=(p(t),p(t))$.

\begin{proof}Denoting by $m$ the spectral parameter of $\hat{\Lambda}$,
observe that the lift $G:=p(m)$ of $\hat{\Lambda}$ is a
$(\D+m\eta)$-parallel section, which is orthogonal to
$\hat{\Lambda}$.
\end{proof}
\end{corollary}

Taking $d=2$, with codimension $1$, we obtain the corresponding
classical result (see \cite{Bia05,Bia05a}).

To conclude this section, we will present the conditions in which a
special isothermic surface of type $d$ admits Darboux transforms
which are special isothermic surfaces of type $d-1$.

\begin{proposition}\label{(n-1)Darbouxs}
Let $(\Lambda,\eta)$ be a special isothermic surface in $S^{n}$ of
type $d\in\mathbb{N}$ with respect to $p(t)$. Suppose that
$(\Lambda,\eta)$ is not a special isothermic surface of type $d-1$.
Then $(\Lambda,\eta)$ admits a Darboux transform with spectral
parameter $m$, which is a special isothermic surface of type $d-1$
with respect to a polynomial $\hat{p}(t)$ such that
$\hat{p}(0)=p(0)$ if and only if $m\in\mathbb{R}^{\times}$
is a repeated root of $(p(t),p(t))$. Furthermore, the Darboux
transform in this situation is the complementary surface $\langle
p(m)\rangle$ of $(\Lambda,\eta)$.
\begin{proof}
Consider a Darboux transform $\hat{\Lambda}$ of $(\Lambda,\eta)$,
with parameter $m$, and assume that $(\hat{\Lambda},\hat{\eta})$ is
a special isothermic surface of type $d-1$, with respect to a
polynomial $\hat{p}(t)$ such that $\hat{p}(0)=p(0)$. We know that
the polynomial $\xi(t)$ satisfying
$\xi(t)=\Gamma_{\hat{\Lambda}}^{\Lambda}(1-\frac{t}{m})(1-\frac{t}{m})\hat{p}(t)$
is a polynomial conserved quantity of $(\Lambda,\eta)$, with degree
$d$. Since $\xi(0)=\hat{p}(0)=p(0)$ and $p(t)$ and $\xi(t)$ are
polynomials of degree $d$, we have necessarily $p(t)=\xi(t)$ because
$d$ is the least degree of the polynomials which are conserved
quantities of $(\Lambda,\eta)$. We now get that $m$ is a repeated
root of $(p(t),p(t))$ because
\begin{equation*}
(p(t),p(t))=(\xi(t),\xi(t))=(1-\frac{1}{m}t)^{2}(\hat{p}(t),\hat{p}(t)).
\end{equation*}
Observe that $p(m)$ is not the zero section. Indeed, if it was, we
would have $p(t)=(t-m)q(t)$, for a certain polynomial $q(t)$ of
degree $d-1$. But then $q(t)$ would be a polynomial conserved
quantity of $(\Lambda,\eta)$, which is absurd. By virtue of $p(m)$
being a non-zero $(\D+m\eta)$-parallel section, we conclude that
$p(m)$ never vanishes. Therefore $p(m)\in\Gamma(\mathcal{L})$.
Consequently, since $p(m)_{\Lambda}=0$, we obtain $\langle
p(m)\rangle+\hat{\Lambda}\subset\mathcal{L}$, and then
$\hat{\Lambda}=\langle p(m)\rangle$.

Conversely, assume that there is a repeated root
$m\in\mathbb{R}^{\times}$ of $(p(t),p(t))$. We already know,
from the first part, that $p(m)$ never vanishes. Consider the null
line subbundle $\hat{\Lambda}=\langle p(m)\rangle$ of
$\underline{\mathbb{R}}^{n+1,1}$. Note that we can assume that
$\Lambda\cap\hat{\Lambda}=\{0\}$ considering a non-empty open subset
of $\Sigma$, if necessary, because if $p(m)\in\Gamma(\Lambda)$, we
would have $\eta p(m)=0$, which would imply that $\D p(m)=-m\eta
p(m)=0$. Hence we would get $\Lambda=\langle p(m)\rangle$ constant,
which is absurd. Furthermore, one can prove that $\langle
p(m)\rangle$ is an immersion (considering again a smaller set, if
necessary), having in mind Proposition \ref{Q symmetric and
trace-free}. We can therefore take the Darboux transform
$\hat{\Lambda}=\langle p(m)\rangle$ of $(\Lambda,\eta)$ (a
complementary surface of $(\Lambda,\eta)$, with respect to $p(t)$).
Considering the polynomial conserved quantity $\hat{p}(t)$ of
$(\hat{\Lambda},\hat{\eta})$ defined by
$\hat{p}(t)=\Gamma_{\Lambda}^{\hat{\Lambda}}(1-\frac{t}{m})p(t)$
which has degree $d$, and satisfies $\hat{p}(0)=p(0)$ and
$(\hat{p}(t),\hat{p}(t))=(p(t),p(t))$, we get that $m$ is a repeated
root of the polynomial
\begin{equation*}
2(\hat{p}(t)_{\Lambda},\hat{p}(t)_{\hat{\Lambda}})+(\hat{p}(t)_{(\Lambda\oplus\hat{\Lambda})^{\perp}},\hat{p}(t)_{(\Lambda\oplus\hat{\Lambda})^{\perp}}).
\end{equation*}
Since
$\hat{p}(m)_{(\Lambda\oplus\hat{\Lambda})^{\perp}}=p(m)_{(\Lambda\oplus\hat{\Lambda})^{\perp}}=0$,
consider the polynomial $a(t)$ for which
$\hat{p}(t)_{(\Lambda\oplus\hat{\Lambda})^{\perp}}=(m-t)a(t)$. It
follows that $m$ is a repeated root of the polynomial
\begin{equation*}
2(1-\frac{1}{m}t)(\hat{p}(t)_{\Lambda},p(t)_{\hat{\Lambda}})+(m-t)^{2}(a(t),a(t)),
\end{equation*}
which guarantees that $m$ is a root of
$(\hat{p}(t)_{\Lambda},p(t)_{\hat{\Lambda}})$. Since
$p(m)_{\hat{\Lambda}}$ is always different from zero, we get that
$\hat{p}(m)_{\Lambda}=0$. Therefore, as a consequence of
\begin{equation*}
\hat{p}(m)=\hat{p}(m)_{\Lambda}+\hat{p}(m)_{\hat{\Lambda}}+\hat{p}(m)_{(\Lambda\oplus\hat{\Lambda})^{\perp}}=0,
\end{equation*}
we will consider the polynomial $s(t)$ of degree $d-1$ such that
$\hat{p}(t)=(1-\frac{1}{m}t)s(t)$, which will be a polynomial
conserved quantity of $(\hat{\Lambda},\hat{\eta})$ satisfying
$s(0)=\hat{p}(0)=p(0)$.
\end{proof}
\end{proposition}

With this last Proposition, we proved that given a special
isothermic surface $(\Lambda,\eta)$ of type $d\in\mathbb{N}$, with
respect to a polynomial $p(t)$, which is not a special isothermic
surface of type $d-1$, the Darboux transforms
$(\hat{\Lambda},\hat{\eta})$ of $(\Lambda,\eta)$ which are special
isothermic surfaces of type $d-1$ with respect to a polynomial with
constant term $p(0)$ are exactly the complementary surfaces of
$(\Lambda,\eta)$ with respect to $p(t)$ such that the spectral
parameters are repeated (non-zero) roots of $(p(t),p(t))$. In
particular, we get that there are at most $d$ Darboux transforms
$(\hat{\Lambda},\hat{\eta})$ of $(\Lambda,\eta)$ under these
conditions. The case in which $0$ is a repeated root of
$(p(t),p(t))$ will be discussed in section \ref{Christoffel section}
(see Propositions \ref{0 repeated root} and \ref{0repeated root2}).

\subsubsection{Bianchi permutability theorem of special isothermic
surfaces}

Let $(\Lambda,\eta)$ be an isothermic surface in $S^{n}$ and let
$(\Lambda_{1},\eta_{1})$ and $(\Lambda_{2},\eta_{2})$ be two Darboux
transforms of $(\Lambda,\eta)$, associated to different spectral
parameters $m_{1}$ and $m_{2}$, respectively. Assume that
$\Lambda_{1}\cap\Lambda_{2}=\{0\}$. The Bianchi permutability
theorem of isothermic surfaces claims that there is an isothermic
surface in $S^{n}$ which is simultaneously a Darboux transform of
$(\Lambda_{1},\eta_{1})$, with spectral parameter $m_{2}$, and a
Darboux transform of $(\Lambda_{2},\eta_{2})$, with parameter
$m_{1}$, namely
\begin{equation*}
\hat{\Lambda}:=\Gamma_{\Lambda}^{\Lambda_{1}}(1-\frac{m_{2}}{m_{1}})(\Lambda_{2})=\Gamma_{\Lambda}^{\Lambda_{2}}(1-\frac{m_{1}}{m_{2}})(\Lambda_{1})=\Gamma_{\Lambda_{2}}^{\Lambda_{1}}(\frac{m_{2}}{m_{1}})(\Lambda).
\end{equation*}
$(\hat{\Lambda},\hat{\eta})$ is isothermic when we consider the
$1$-form
\begin{equation*}
\hat{\eta}:=\Gamma_{\Lambda_{2}}^{\Lambda_{1}}(\frac{m_{2}}{m_{1}})\circ\eta\circ
\Gamma_{\Lambda_{2}}^{\Lambda_{1}}(\frac{m_{2}}{m_{1}})^{-1}.
\end{equation*}
In fact this $\hat{\eta}$ is the closed $1$-form associated to
$\hat{\Lambda}$ as a Darboux transform of $(\Lambda_{1},\eta_{1})$
with parameter $m_{2}$, and associated to $\hat{\Lambda}$ as a
Darboux transform of $(\Lambda_{2},\eta_{2})$ with parameter
$m_{1}$.

With the above notations, we have the following Theorem.

\begin{theorem}\label{Bianchi permutability theorem}If $(\Lambda,\eta)$ is a special isothermic surface in $S^{n}$
of type $d\in\mathbb{N}_{0}$, with respect to a polynomial $p(t)$
for which $p(m_{1})\in\Gamma(\Lambda_{1}^{\perp})$ and
$p(m_{2})\in\Gamma(\Lambda_{2}^{\perp})$, then
$(\hat{\Lambda},\hat{\eta})$ is a special isothermic surface of type
$d$, with respect to a polynomial $\hat{p}(t)$ for which
$\hat{p}(0)=p(0)$ and $(\hat{p}(t),\hat{p}(t))=(p(t),p(t))$.
\begin{proof}
Assume that $(\Lambda,\eta)$ is a special isothermic surface of type
$d\in\mathbb{N}_{0}$, with respect to a polynomial $p(t)$ for which
$p(m_{1})\in\Gamma(\Lambda_{1}^{\perp})$ and
$p(m_{2})\in\Gamma(\Lambda_{2}^{\perp})$. Theorem \ref{general case}
guarantees that the polynomial $p_{1}(t)$ such that
$p_{1}(t)=\Gamma_{\Lambda}^{\Lambda_{1}}(1-\frac{t}{m_{1}})p(t)$,
for all $t\in\mathbb{R}\backslash\{m_{1}\}$, is a polynomial
conserved quantity of $(\Lambda_{1},\eta_{1})$, with degree $d$,
satisfying $p_{1}(0)=p(0)$ and $(p_{1}(t),p_{1}(t))=(p(t),p(t))$. It
follows from $p(m_{2})\in\Gamma(\Lambda_{2}^{\perp})$ that
\begin{equation*}
p_{1}(m_{2})=\Gamma_{\Lambda}^{\Lambda_{1}}(1-\frac{m_{2}}{m_{1}})p(m_{2})\in\big(\Gamma_{\Lambda}^{\Lambda_{1}}(1-\frac{m_{2}}{m_{1}})(\Lambda_{2})\big)^{\perp}=\hat{\Lambda}^{\perp},
\end{equation*}
which implies, using again Theorem \ref{general case}, that
$(\hat{\Lambda},\hat{\eta})$ is a special isothermic surface of type
$d$, with respect to a polynomial $\hat{p}(t)$, namely the one
satisfying
$\hat{p}(t)=\Gamma_{\Lambda_{1}}^{\hat{\Lambda}}(1-\frac{t}{m_{2}})p_{1}(t)$,
such that $\hat{p}(0)=p_{1}(0)$ and
$(\hat{p}(t),\hat{p}(t))=(p_{1}(t),p_{1}(t))$. Consequently
$\hat{p}(0)=p(0)$ and $(\hat{p}(t),\hat{p}(t))=(p(t),p(t))$.
\end{proof}
\end{theorem}

As a consequence of Theorems \ref{Bianchi permutability theorem} and
\ref{the reverse}, we obtain the following Corollary. One can find
the corresponding classical result in \cite{Bia05}.

\begin{corollary}If $(\Lambda,\eta)$ is a special isothermic surface in $S^{3}$
of type $d\in\mathbb{N}$, with respect to a polynomial $p(t)$, but
it is not a special isothermic surface of type $d-1$, and if
$(\Lambda_{1},\eta_{1})$ and $(\Lambda_{2},\eta_{2})$ are special
isothermic surfaces of type $d$, with respect to polynomials
$p_{1}(t)$ and $p_{2}(t)$, respectively, such that
$p_{1}(0)=p_{2}(0)=p(0)$, then $(\hat{\Lambda},\hat{\eta})$ is a
special isothermic surface of type $d$, with respect to a polynomial
$\hat{p}(t)$ for which $\hat{p}(0)=p(0)$ and
$(\hat{p}(t),\hat{p}(t))=(p(t),p(t))$.
\end{corollary}

\subsection{Christoffel transforms of a special isothermic
surface}\label{Christoffel section}

Let $(\Lambda,\eta)$ be an isothermic surface in $S^{n}$. Fix a pair
$(v_{\infty},v_{0})$ where $v_{\infty},v_{0}\in\mathcal{L}$ and
$(v_{0},v_{\infty})=-1$.

\begin{theorem}\label{theorem-Christoffel}If $(\Lambda,\eta)$ is a special isothermic surface in $S^{n}$
of type $d\in\mathbb{N}_{0}$ with respect to a polynomial $p(t)$ for
which the constant term $p(0)\in\langle v_{\infty}\rangle^{\perp}$,
then $(\Lambda^{c},\eta^{c})$ is a special isothermic surface of
type $d$, with respect to a polynomial $q(t)$ such that
$(q(t),q(t))=(p(t),p(t))$ and $q(0)\in\langle v_{0}\rangle^{\perp}$.
Moreover, if $p(0)\in\langle v_{\infty}\rangle$, then
$q(0)\in\langle v_{0}\rangle$.
\begin{proof}
Assuming that a polynomial
$p(t)\in\Gamma(\underline{\mathbb{R}}^{n+1,1})[t]$ satisfies
$p(0)\in\langle v_{\infty}\rangle^{\perp}$, we can define the
polynomial $q(t)$ by $q(t)=\Gamma^{c}(t)p(t)$, for all $t\neq0$ (recall
section \ref{Christoffel transforms}). If $p(t)$ is
$(\D+t\eta)$-parallel, we automatically obtain that $q(t)$ is
$(\D+t\eta^{c})$-parallel. Furthermore,
\begin{equation*}
q(0)=-(p_{1},v_{\infty})v_{0}+p_{0}+(p_{0},v_{0})v_{\infty}+(p_{0},f^{c})v_{0},
\end{equation*}
where $p(t)=\sum_{k=0}^d p_{k}t^{k}$, and
$f^{c}:\Sigma\to\mathbb{R}^{n}=\langle
v_{0},v_{\infty}\rangle^{\perp}$ is an immersion such that
$\eta:=\Ad(\exp(f\wedge v_{\infty}))\,(\D f^{c}\wedge v_{0})$.
\end{proof}
\end{theorem}

The last part of this Theorem gives in particular the classical
result, due to Bianchi \cite{Bia05}, which states that given a
special isothermic surface $f$ in $\mathbb{R}^{3}$ of class
$(A,B,C,D)$, with conformal curvature line coordinates $(u,v)$, the
Christoffel transforms $f^{c}$ of $f$ such that
\begin{equation*}
f^{c}_{u}=e^{-2\theta}f_{u}\mbox{ and }f^{c}_{v}=-e^{-2\theta}f_{v},
\end{equation*}
where the first fundamental form of $f$ is $I=e^{2\theta}(\D
u^{2}+\D v^{2})$, are special isothermic surfaces of class
$(A,C,B,D)$. Indeed, fix
$v_{\infty}\in\mathcal{L}\subseteq\mathbb{R}^{4,1}$ and $v_{0}\in
E(v_{\infty})$, and identify $\mathbb{R}^{3}$ with $\langle
v_{0},v_{\infty}\rangle^{\perp}$. Consider the surfaces
$F=\exp(f\wedge
v_{\infty})v_{0}$ in $E(v_{\infty})$ and
$F^{c}=\exp(f^{c}\wedge
v_{0})v_{\infty}$ in $E(v_{0})$. Take
now the isothermic surfaces $(\Lambda:=\langle F\rangle,\eta)$ and $(\Lambda^{c}:=\langle
F^{c}\rangle,\eta^{c})$ where $\eta=\Ad(\exp(f\wedge v_{\infty}))\;(\D f^{c}\wedge
v_{0})$ and $\eta^{c}=\Ad(\exp(f^{c}\wedge v_{0}))\;(\D f\wedge
v_{\infty})$.

Consider the unit vector fields $W_{1}=e^{-\theta}F_{u}$ and
$W_{2}=e^{-\theta}F_{v}$, and let $N$ be a unit normal to $F$ in $E(v_{\infty})$. Take the
unit vector fields $W_{1}^{c}:=e ^{\theta}F^{c}_{u}=\Gamma^{c}(t)W_{1}$, $W_{2}^{c}:=e^{\theta}F^{c}_{v}=-\Gamma^{c}(t)W_{2}$ and
$N^{c}=-\Gamma^{c}(t)N$ of $F^{c}$.

By virtue of $F$ being a special isothermic surface of class
$(A,B,C,D)$, we learn that $p(t)=p_{0}+p_{1}t+p_{2}t^{2}$ given by
\begin{equation*}
\begin{split}
&p_{0}=Bv_{\infty},\\
&p_{1}=-CF-e^{\theta}H_{u}W_{1}+e^{\theta}H_{v}W_{2}-(\frac{1}{2}M+A)N+\frac{1}{2}Lv_{\infty}\;\mbox{
and }\\
&p_{2}=HF+N,
\end{split}
\end{equation*}
is a polynomial conserved quantity of $(\Lambda,\eta)$ (in fact
assuming that $H_{u}H_{v}$ never vanishes). Considering now the
polynomial $q(t)=q_{0}+q_{1}t+q_{2}t^{2}$ for which $q(t)=\Gamma^{c}(t)p(t)$, for all $t\neq0$, we get
\begin{equation*}
\begin{split}
&q_{0}=-Cv_{0},\\
&q_{1}=BF^{c}-e^{\theta}H_{u}W_{1}^{c}-e^{\theta}H_{v}W_{2}^{c}+(\frac{1}{2}M+A)N^{c}+Hv_{0}\;\mbox{
and }\\
&q_{2}=\frac{1}{2}LF^{c}-N^{c}.
\end{split}
\end{equation*}
The result follows noticing that $H^{c}=-\frac{1}{2}L$, $L^{c}=-2H$
and $M^{c}=M$ (with the obvious notations).

\begin{proposition}\label{0 repeated root}If $(\Lambda,\eta)$ is a special isothermic surface in $S^{n}$
of type $d\in\mathbb{N}$, with respect to a polynomial $p(t)$ for
which $0\neq p(0)\in\langle v_{\infty}\rangle$ and $0$ is a repeated
root of $(p(t),p(t))$, then $(\Lambda^{c},\eta^{c})$ is a special
isothermic surface of type $d-1$.
\begin{proof}
Assume the conditions of the hypothesis with $p(t)=\sum_{k=0}^d
p_{k}t^{k}$. Take the $d$-degree polynomial $q(t)$ defined by
$\Gamma^{c}(t)p(t)$, for all $t\neq0$, which is a conserved quantity
of $(\Lambda^{c},\eta^{c})$. Since $0$ is a repeated root of
$(p(t),p(t))$, we get $(p_{0},p_{1})=0$, which implies that
$(p_{1},v_{\infty})=0$, and then
$q(0)=-(p_{1},v_{\infty})v_{0}=0$. Considering the polynomial
$s(t)$ for which $q(t)=ts(t)$, we obtain a polynomial conserved
quantity of $(\Lambda^{c},\eta^{c})$, with degree equal to $d-1$.
\end{proof}
\end{proposition}

\begin{proposition}\label{0repeated root2}
If $(\Lambda,\eta)$ is a special isothermic surface in $S^{3}$ of
type $d\in\mathbb{N}$ with respect to a polynomial $p(t)$, but it is
not a special isothermic surface of type $d-1$, and if
$(\Lambda^{c},\eta^{c})$ is a special isothermic surface of type
$d-1$, then $0$ is a repeated root of $(p(t),p(t))$ and
$p(0)\in\langle v_{\infty}\rangle$.
\begin{proof}
Assume the conditions of the hypothesis. Denote by $s(t)$ a
polynomial of degree $d-1$ which is a polynomial conserved quantity
of $(\Lambda^{c},\eta^{c})$. Then the polynomial $g(t)$ defined by
$g(t)=\Gamma^{c}(t)^{-1}(ts(t))$, for all $t\neq0$, is a polynomial
conserved quantity of $(\Lambda,\eta)$, with degree $d$. Using the
fact that $(\Lambda,\eta)$ is not a special isothermic surface of
type $d-1$ and the fact that the codimension is $1$, we obtain
$p(t)=\alpha g(t)$, for some non-zero constant $\alpha$. Hence we
conclude that
$(p(t),p(t))=\alpha^{2}(g(t),g(t))=\alpha^{2}t^{2}(s(t),s(t))$,
which guarantees that $0$ is a repeated root of $(p(t),p(t))$.
Finally, we get that $p(0)\in\langle v_{\infty}\rangle$, because the
constant term of $ts(t)$ belongs to $\langle v_{0}\rangle$ (it is
equal to zero).
\end{proof}
\end{proposition}

\begin{remark}
Consider a special isothermic surface $(\Lambda,\eta)$ in $S^{n}$ of
type $d\in\mathbb{N}$, with respect to a polynomial $p(t)$. If
$p_{0}:=p(0)\in\mathcal{L}$, we can consider $p_{0}$ as the point at
infinity $v_{\infty}$, and apply all these results. Note that
$p_{0}\in\mathcal{L}$ if and only if $p_{0}\neq0$ and $0$ is a
root of $(p(t),p(t))$.

Observe that if $(p(t),p(t))$ admits a repeated root $m$, we can
therefore conclude, in general, that $(\Lambda,\eta)$ admits Darboux
transform(s) (which will be complementary surfaces) or Christoffel
transforms which are special isothermic surfaces of type $d-1$.
\end{remark}

\subsection{T-transforms of a special isothermic surface}

\begin{theorem}\label{T-transforms}If $(\Lambda,\eta)$ is a special isothermic surface in $S^{n}$
of type $d\in\mathbb{N}_{0}$ with respect to a polynomial $p(t)$,
then $(\Lambda_{s},\eta_{s})$ is a special isothermic surface of
type $d$, with respect to a polynomial with constant term equal to
$\Phi_{s}(p(s))$.
\begin{proof}
Given a polynomial conserved quantity $p(t)$ of $(\Lambda,\eta)$,
$p(t+s)$ is $(\D+(t+s)\eta)$-parallel, which guarantees that
$q(t):=\Phi_{s}p(t+s)$ is parallel with respect to
$\D+t\eta_{s}=\Phi_{s}\cdot(\D+(t+s)\eta)$.
\end{proof}
\end{theorem}

The corresponding classical result of $T$-transforms of a special
isothermic surfaces can be found in \cite{Bia05a}, which states that
the $T$-transforms of a special isothermic surface are special
isothermic surfaces, but in general of a different class and in a
different space-form.

\medskip

Lawson correspondence is a consequence of Theorem \ref{T-transforms}
when we consider the case of special isothermic surfaces in $S^{3}$
of type $1$, which is related, as we know, to constant mean
curvature surfaces in a $3$-dimensional space-form. Indeed, consider
an arbitrary full isothermic surface $(\Lambda,\eta)$ in $S^{3}$
which is a constant mean curvature surface in a space-form $E(w)$
(with sectional curvature equal to $K:=-(w,w)$). Denote by $F$ the
lift of $\Lambda$ which lives in $E(w)$, by $N$ a unit normal to $F$ in $E(w)$
and finally by $H$ the (constant) mean curvature of $F$ with respect to $N$. We have
that $p(t):=w+p_{1}t$, with $p_{1}=HF+N$, is a polynomial conserved
quantity of $(\Lambda,\eta)$, considering $\eta=F\wedge \D p_{1}$.

Given an arbitrary $T$-transform $(\Lambda_{s},\eta_{s})$ of
$(\Lambda,\eta)$ (associated to a $\Phi_{s}$), take the polynomial
conserved quantity
\begin{equation*}
q(t):=\Phi_{s}p(t+s)=\Phi_{s}(p(s))+\Phi_{s}(p_{1})t
\end{equation*}
of $(\Lambda_{s},\eta_{s})$. Therefore we conclude that
$\Lambda_{s}$ is a constant mean curvature surface in the space-form
$E(\Phi_{s}(p(s)))$, which has sectional curvature
$K_{s}=-(\Phi_{s}(p(s)),\Phi_{s}(p(s)))$, with (constant) mean
curvature equal to $H_{s}:=-(\Phi_{s}(p_{1}),\Phi_{s}(p(s)))$, with respect to $\Phi_{s}(N+sF)$. Since
\begin{equation*}
(q(t),q(t))=(\Phi_{s}p(t+s),\Phi_{s}p(t+s))=(p(t+s),p(t+s)),
\end{equation*}
we obtain, in particular, that the discriminants of both second
degree polynomials
\begin{equation*}
(q(t),q(t))=-K_{s}-2H_{s}t+t^{2} \mbox{ and
}(p(t),p(t))=-K-2Ht+t^{2}
\end{equation*}
are equal. Consequently, $4H_{s}^{2}+4K_{s}=4H^{2}+4K$, i.e.,
\begin{equation*}
H_{s}^{2}+K_{s}=H^{2}+K.
\end{equation*}

\section{Spherical system and sphere-planes}\label{spherical system and sphere-planes}

\subsection{Spherical system and sphere-planes of complementary surfaces}

An isothermic surface $(\Lambda,\eta)$ in $S^{n}$ together with a
Darboux transform $\hat{\Lambda}$ define two
interesting sphere congruences.  First, $\Lambda,\hat{\Lambda}$
envelop the congruence of $2$-spheres given by
\begin{equation*}
V_{R}:=\Lambda^{(1)}\oplus\hat{\Lambda}=\Lambda\oplus\hat{\Lambda}^{(1)}.
\end{equation*}
Second, define the \emph{spherical system} of $\Lambda$ and
$\hat{\Lambda}$ as the $(n-2)$-sphere congruence
\begin{equation*}
\mathcal{C}:=\Lambda\oplus\hat{\Lambda}\oplus
V_{R}^{\perp}=(\Lambda^{(1)})^{\perp}\oplus\hat{\Lambda}. 
\end{equation*}
This is the family of codimension $2$ spheres that cut both $\Lambda$
and $\hat{\Lambda}$ orthogonally at each point.  Orthoprojection of
$\D$ onto $\mathcal{C}$ defines a connection $\mathcal{D}$ which is
flat: indeed, $\Lambda\oplus\hat{\Lambda}$ and $V_R$ are parallel
subbundles on which $\mathcal{D}$ is flat, $V_R$ since $\Lambda$ has
flat normal bundle and $\Lambda\oplus\hat{\Lambda}$ since
$\mathcal{D}$ is also the orthoprojection of $\D+m\eta$ with respect
to which $\hat{\Lambda}$ is parallel by definition.  In codimension
$1$, this flatness amounts to the assertion that $\mathcal{C}$
defines a \emph{cyclic system} of circles, whence our terminology.

Finally, let $w\in\mathbb{R}_{\times}^{n+1,1}$ and set
\begin{equation*}
\mathcal{P}:=\mathcal{C}+\langle w\rangle.
\end{equation*}
On the open set where $w\notin \mathcal{C}$, $\mathcal{P}$ defines a
congruence of hyperspheres we call the \emph{sphere-planes
congruence} of $\Lambda$ and $\hat{\Lambda}$ with respect to $w$.
Geometrically, $\mathcal{P}$ is the unique congruence of totally
geodesic hyperspheres in $E(w)$ containing the spherical system.

Darboux \cite{Dar99b} proves that an isothermic surface in $\R^3$ is
special in the classical sense if and only if it admits a pair of
possibly complex conjugate Darboux transforms for which the
corresponding circle-planes coincide and, in this case, those Darboux
transforms are complementary surfaces.  This all goes through in
arbitrary codimension and for arbitrary space-forms at least in the
case of real sphere-planes.

\begin{theorem}\label{theorem of sphere-planes}
Let $(\Lambda,\eta)$ be a full isothermic surface in $S^{n}$, 
$n\geq3$,  with two Darboux transforms
$\hat{\Lambda}_{1}$ and $\hat{\Lambda}_{2}$ of $(\Lambda,\eta)$
of parameters $m_1\neq m_2$, such that there is a
$w\in\mathbb{R}^{n+1,1}\backslash\Lambda^{\perp}_{x}$, for all $x$,
which never belongs to the spherical systems of $\hat{\Lambda}_{1}$
and $\hat{\Lambda}_{2}$ \footnote{that is, the spherical systems of the
two Darboux pairs $\Lambda$ and $\hat{\Lambda}_{1}$, and $\Lambda$
and $\hat{\Lambda}_{2}$.}. If the sphere-planes of
$\hat{\Lambda}_{1}$ and $\hat{\Lambda}_{2}$ associated to $w$
coincide and contain no principal direction of $\Lambda$, then
$(\Lambda,\eta)$ is a special isothermic surface of type $2$, with
respect to a polynomial $p(t)$ with 
$p_0\in\langle w\rangle$, and such that $\hat{\Lambda}_{1}$ and
$\hat{\Lambda}_{2}$ are complementary surfaces of $(\Lambda,\eta)$
with respect to $p(t)$.

Conversely, if $(\Lambda,\eta)$ is special isothermic of type $2$ in $E(w)$
with complementary surfaces $\hat{\Lambda}_1,\hat{\Lambda}_2$ then
the sphere-planes through each $\hat{\Lambda}_i$ and $w$ coincide.
\begin{proof}
Let $F_{i}$ be a $(\D+m_{i}\eta)$-parallel lift of
$\hat{\Lambda}_{i}$. We assume that the sphere-planes coincide:
\begin{equation*}
\mathcal{P}=\Lambda\oplus\hat{\Lambda}_{1}\oplus V_{1}^{\perp}\oplus\langle
w\rangle=\Lambda\oplus\hat{\Lambda}_{2}\oplus
V_{2}^{\perp}\oplus\langle w\rangle,
\end{equation*}
where $V_{i}$ is the enveloped sphere congruence of $\Lambda$ and
$\hat{\Lambda}_{i}$.  We can therefore write
\begin{equation}\label{equation of the motivation}
F_{1}=\xi+\beta F_{2}+Q
\end{equation}
where $\beta$ is a function, $\xi\in\Gamma(\Lambda^{(1)})^{\perp}$
and $Q\in\Gamma\langle w\rangle$.  We are going to prove that $\beta$
and $Q$ are constant.

As in \eqref{eq:2}, for a lift $F\in\Gamma\Lambda$, we write
\begin{equation*}
\eta=e^{-2\theta}F\wedge(-F_u\D u+F_v\D v),
\end{equation*}
for curvature line coordinates $u,v$.  Now apply $\del/\del
u+m_1\eta_{\del/\del u}$ to \eqref{equation of the motivation} to get
\begin{equation*}
0=\xi_u+\beta_uF_2+(m_1-m_2)\beta\eta_{\del/\del u}F_2+Q_u+m_1\eta_{\del/\del u}Q.
\end{equation*}
Let $\pi$ denote orthoprojection onto $V_2$ and set
$\pi^{\perp}=1-\pi$.  We apply $\pi$ to the last equation and
rearrange to get:
\begin{equation*}
\pi(\xi_u)+(m_1-m_2)\beta\eta_{\del/\del u}F_2+m_1\eta_{\del/\del u}Q=-\beta_uF_2-Q_u+\pi^{\perp}(Q_{u}).
\end{equation*}
Here the left hand side takes values in $\Lambda_{u}=\langle
F,F_{u}\rangle$ since $u,v$ are curvature line coordinates, while the
right hand side lies in $\mathcal{P}$.  Our assumption on principal
directions is that $\Lambda_u\cap \mathcal{P}=\Lambda$ from which we
deduce that $Q_u$ takes values in
$\Lambda\oplus\hat{\Lambda}_{2}\oplus V_{2}^{\perp}$.  However, this
last has trivial intersection with $\langle w\rangle$ so that
$Q_u=0$.  Since $F_{2}$ does not lie in $\Lambda$, we now conclude
that $\beta_u=0$ also.   Similarly, we obtain $Q_{v}=0$ and $\beta_{v}=0$.

Now let $p(t)=p_{0}+p_{1}t+p_{2}t^{2}$ be the polynomial 
such that
\begin{equation*}
p(m_{1})=m_{1}F_{1},\;p(m_{2})=m_{2}\beta F_{2}\mbox{ and
}p(0)=\frac{m_1m_2}{m_2-m_1}Q.
\end{equation*}
One readily computes from \eqref{equation of the motivation} that
$p_{2}=\xi/(m_1-m_2)\in\Gamma(\Lambda^{(1)})^{\perp}$ so that
$\eta p_{2}=0$.  Thus $(\D+t\eta)p(t)$ is quadratic in $t$ with zeros
at $m_{1},m_2$ and $0$ and so vanishes identically.  We conclude that
$p(t)$ is our desired polynomial conserved quantity.

For the converse, if $\Lambda$ has a polynomial conserved quantity
$p(t)=p_{0}+p_{1}t+p_{2}t^{2}$ and $\hat{\Lambda}=\langle
p(m)\rangle$ is a complementary surface then the corresponding
spherical system is $\mathcal{C}=(\Lambda^{(1)})^{\perp}\oplus\langle
p(m)\rangle$ so that the sphere-planes are given by
\begin{equation*}
\mathcal{C}+\langle w\rangle=(\Lambda^{(1)})^{\perp}+\langle
p_{0}+p_{1}m+p_{2}m^{2}\rangle+\langle w\rangle=
(\Lambda^{(1)})^{\perp}+\langle w,p_{1}\rangle,
\end{equation*}
which is independent of $m$.
\end{proof}
\end{theorem}

A similar but simpler argument shows that the stronger
condition of coincident spherical systems characterises those special
isothermic surfaces of type $1$ that admit two complementary
surfaces\footnote{In codimension $1$, these are the surfaces of
constant mean curvature $H$ in a space-form of sectional curvature
$K$ for which $H^2+K>0$.}:
\begin{theorem}
Let $(\Lambda,\eta)$ be a full isothermic surface in $S^{n}$, 
$n\geq3$,  with two Darboux transforms
$\hat{\Lambda}_{1}$ and $\hat{\Lambda}_{2}$ of $(\Lambda,\eta)$
of parameters $m_1\neq m_2$.   The spherical systems through 
$\hat{\Lambda}_{1}$ and $\hat{\Lambda}_{2}$ coincide if and only if
$\Lambda$ is special isothermic of type $1$ with the
$\hat{\Lambda}_i$ complementary surfaces.
\end{theorem}
\begin{proof}
If the spherical systems coincide, we have, with $F_i$
$(\D+m_i\eta)$-parallel lifts of $\hat{\Lambda}_i$,
\begin{equation*}
F_1=\xi+\beta F_2,
\end{equation*}
for a function $\beta$ and $\xi\in\Gamma(\Lambda^{(1)})^{\perp}$.
Applying $\D+m_1\eta$ to this yields
\begin{equation*}
0=\D\xi+\D\beta F_2+(m_1-m_2)\eta F_2\equiv\D\beta F_2\mod\Lambda^{\perp}
\end{equation*}
so that $\beta$ is constant.  Now define $p(t)=p_0+p_{1}t$ by
requiring that 
\begin{equation*}
p(m_1)=F_1,\qquad p(m_2)=\beta F_2.
\end{equation*}
One computes that $p_1=\xi/(m_1-m_2)$ so that $\eta p_1=0$. Now
$(\D+t\eta)p(t)$ is first order in $t$ with zeros at $m_1,m_2$ and so
vanishes.

The converse is straightforward.
\end{proof}

\subsection{Enveloping surfaces of circle-planes of complementary
surfaces}

Darboux \cite[p.~507]{Dar99b} (see also \cite[\S15]{Bia05}) proves
that, for a special isothermic surface in $\R^{3}$,
the enveloping surface of the congruence of circle-planes of the
complementary surfaces is isometric to a quadric in
$\mathbb{R}^{2,1}$. We now prove a version of this result in arbitrary
codimension.

Let $(\Lambda,\eta)$ be a special isothermic surface of type $2$ in
$\R^n=E(w)$, $w\in\mathcal{L}$, with respect to the polynomial
$p(t)=p_{0}+p_{1}t+p_{2}t^{2}$. Assume that $p_{2}$ is a unit section
and write $p_{0}=Bw$ for a constant $B$ (possibly zero). Let
$\hat{\Lambda}=\langle p(m)\rangle$ be a complementary surface of
$(\Lambda,\eta)$.

Contemplate the congruence of circles
$\mathcal{C}=\Lambda\oplus\hat{\Lambda}\oplus\langle p_{2}\rangle$
which cuts both $\Lambda$ and $\hat\Lambda$ orthogonally in the
direction defined by $p_2$.  The congruence
$\mathcal{P}=\mathcal{C}\oplus\langle w\rangle$ of planes through
these circles is defined whenever $w\notin\mathcal{C}$.

\begin{theorem}
\item{}
\begin{enumerate}
\item The congruence $\mathcal{P}$ generically admits an enveloping
surface $\Lambda_{e}$.
\item The section of $\Lambda_{e}$ which takes values in $E(w)$ is
locally isometric to a quadric in $\mathbb{R}^{2,1}$.
\end{enumerate}
\end{theorem}

\begin{proof}
Consider sections of $\mathcal{P}$ of the form $G=\alpha p_{2}+\beta
p(m)+sw=G_1+sw$ for functions $\alpha,\beta,s$.  Such a section will span an
enveloping surface if it is null and $\D G_{1}$ takes values in
$\mathcal{P}$.  However, $\D p_{2}$ is proportional to $\D
p(m)=-m\eta\,p(m)$ modulo $\Lambda$ (the trace-free second
fundamental form of $\Lambda$ is proportional to the holomorphic
quadratic differential $q$) so that $\alpha,\beta$ can be chosen to
ensure that $\D G_1\in\Omega^1(\mathcal{C})$.  Now $s$ can be
chosen to make $G$ null so long as $(G_1,w)$ is non-zero.

In this situation, we normalise so that $(G,w)=-1$ and let
$\Lambda_e$ be the span of $G$.  We now show that $G$ is isometric to
a quadric.

Observe that $p(m)$ is a $\mathcal{D}$-parallel section of
$\mathcal{C}$ where $\mathcal{D}$ is the flat connection on
$\mathcal{C}$ given by the orthoprojection of $\D$ onto $\mathcal{C}$
(indeed, $\mathcal{C}$ is a parallel subbundle of the spherical
system through $p(m)$).  Now construct a parallel frame of
$\mathcal{C}$ by taking $F\in\Gamma\Lambda$ with $(F,p(m))=-1$ and
setting $Z:=(p_{2},p(m))F+p_{2}$: a unit section of $\mathcal{C}$
orthogonal to $\Lambda$ and $\hat\Lambda$. We have
\begin{equation*}
G_{1}=\alpha p_{2}+\beta p(m)=\alpha (Z-(p_{2},p(m))F)+\beta
p(m)=\alpha Z+\beta p(m)+\gamma F,
\end{equation*}
where
\begin{equation}
\label{eq:12}
\gamma=-\alpha(p_{2},p(m)).
\end{equation}

Since $(G,w)=-1$, equivalently, $(G_{1},w)=-1$, we obtain
\begin{equation*}
\alpha(p_{2},w)+\beta(p(m),w)=-1,
\end{equation*}
so that
\begin{equation*}
\alpha^{2}(p_{2},p_{0})+\alpha\beta (p(m),p_{0})+\alpha B=0.
\end{equation*}
Now substitute \eqref{eq:12} to eliminate the non-constant inner
product $(p_2,p_0)$ and deduce
\begin{multline}
\label{eq:13}
-\bigl(m^2+m(p_2,p_1)\bigr)\alpha^2
+\bigl(m(p_{1},p_{0})-m^4-m^3(p_2,p_1)\bigr)\alpha\beta\\
-\alpha\gamma-m^{2}\beta\gamma+B\alpha=0.
\end{multline}
Note that all the inner products $(p_i,p_j)$ occurring in \eqref{eq:13} are
constant.

Since $\mathcal{D}$ is flat, there are local gauge transformations
$\Psi:(\mathcal{C},\mathcal{D})\cong(\underline{\R}^{2,1},\D)$ and we
see that $g=\Psi\circ G_1$ has image in the quadric defined by
\eqref{eq:13}.  Moreover, since $(\D G_{1},w)=0$ and $\D
G_{1}\in\Omega^{1}(\mathcal{C})$, we get
\begin{equation*}
(\D G,\D G)=(\D G_{1},\D
G_{1})=(\mathcal{D}G_{1},\mathcal{D}G_{1})=(\D g,\D g)
\end{equation*}
so that $\Lambda_e$, with the metric induced from $E(w)$, is
isometric to the quadric \eqref{eq:13} wherever it immerses.
\end{proof}

% \bibliographystyle{amsplain}
% \bibliography{master}

\begin{thebibliography}{10}

\bibitem{Bia05}
L.~Bianchi, \emph{Ricerche sulle superficie isoterme e sulla deformazione delle
  quadriche}, Ann. di Mat. \textbf{11} (1905), 93--157.

\bibitem{Bia05a}
\bysame, \emph{Complementi alle ricerche sulle superficie isoterme}, Ann. di
  Mat. \textbf{12} (1905), 19--54.

\bibitem{BurHerPedPin97}
F.~Burstall, U.~Hertrich-Jeromin, F.~Pedit, and U.~Pinkall, \emph{Curved flats
  and isothermic surfaces}, Math. Z. \textbf{225} (1997), no.~2, 199--209.
  \MR{MR1464926 (98j:53004)}

\bibitem{Bur06}
F.E. Burstall, \emph{Isothermic surfaces: conformal geometry, {C}lifford
  algebras and integrable systems}, Integrable systems, geometry, and topology,
  AMS/IP Stud. Adv. Math., vol.~36, Amer. Math. Soc., Providence, RI, 2006,
  pp.~1--82. \MR{MR2222512 (2008b:53006)}

\bibitem{BurCal}
F.E. Burstall and D.~Calderbank, \emph{Conformal submanifold geometry}, In
  preparation.

\bibitem{BurDonPedPin09}
F.E. Burstall, N.M. Donaldson, F.~Pedit, and U.~Pinkall, \emph{Isothermic
  submanifolds of symmetric {${R}$}-spaces}, J. reine u. Angew. Math. (to
  appear).

\bibitem{Cal03}
P.~Calapso, \emph{Sulle superficie a linee di curvatura isoterme}, Rendiconti
  Circolo Matematico di Palermo \textbf{17} (1903), 275--286.

\bibitem{Cal15}
\bysame, \emph{Sulle trasformazioni delle superficie isoterme}, Ann. di Mat.
  \textbf{24} (1915), 11--48.

\bibitem{Chr67}
E.~Christoffel, \emph{Ueber einige allgemeine {E}igenshaften der
  {M}inimumsfl\"achen}, Crelle's J. \textbf{67} (1867), 218--228.

\bibitem{CieGolSym95}
J.~Cie{\'s}li{\'n}ski, P.~Goldstein, and A.~Sym, \emph{Isothermic
  surfaces in {$\mathbf{E}^3$} as soliton surfaces}, Phys. Lett. A \textbf{205}
  (1995), no.~1, 37--43. \MR{MR1352426 (96g:53005)}

\bibitem{Dar99e}
G.~Darboux, \emph{Sur les surfaces isothermiques}, C.R. Acad. Sci. Paris
  \textbf{128} (1899), 1299--1305, 1538.

\bibitem{Dar99b}
\bysame, \emph{Sur les surfaces isothermiques}, Ann. Sci. \'Ecole Norm.
  Sup. (3) \textbf{16} (1899), 491--508. \MR{MR1508975}

\bibitem{Eis62}
L.P. Eisenhart, \emph{Transformations of surfaces}, Second edition,
  Chelsea Publishing Co., New York, 1962. \MR{MR0142061 (25 \#5455)}

\bibitem{Ger31}
S.~Germain, \emph{M\'{e}moire sur la coubure des surfaces}, Crelle's J.
  \textbf{7} (1831), 1--29.

\bibitem{Her97}
U.~Hertrich-Jeromin, \emph{Supplement on curved flats in the space of point
  pairs and isothermic surfaces: a quaternionic calculus}, Doc. Math.
  \textbf{2} (1997), 335--350 (electronic). \MR{MR1487468 (99b:53017)}

\bibitem{Her03}
\bysame, \emph{Introduction to {M}\"obius differential geometry}, London
  Mathematical Society Lecture Note Series, vol. 300, Cambridge University
  Press, Cambridge, 2003. \MR{MR2004958 (2004g:53001)}

\bibitem{HerPed97}
U.~Hertrich-Jeromin and F.~ Pedit, \emph{Remarks on the {D}arboux transform
  of isothermic surfaces}, Doc. Math. \textbf{2} (1997), 313--333 (electronic).
  \MR{MR1487467 (99k:53006)}

\bibitem{KamPedPin98}
G.~Kamberov, F.~Pedit, and U.~Pinkall, \emph{Bonnet pairs and
  isothermic surfaces}, Duke Math. J. \textbf{92} (1998), no.~3, 637--644.
  \MR{MR1620534 (99h:53009)}

\bibitem{Pal88}
B.~Palmer, \emph{Isothermic surfaces and the {G}auss map}, Proc. Amer.
  Math. Soc. \textbf{104} (1988), no.~3, 876--884. \MR{MR964868 (90a:53077)}

\bibitem{San08}
S.D. Santos, \emph{Special isothermic surfaces}, Ph.D. thesis, PhD, University
  of Bath, 2008.

\bibitem{Sch01}
W.K. Schief, \emph{Isothermic surfaces in spaces of arbitrary dimension:
  integrability, discretization, and {B}\"acklund transformations---a discrete
  {C}alapso equation}, Stud. Appl. Math. \textbf{106} (2001), no.~1, 85--137.
  \MR{MR1805487 (2002k:37140)}

\end{thebibliography}
\providecommand{\bysame}{\leavevmode\hbox to3em{\hrulefill}\thinspace}
\providecommand{\MR}{\relax\ifhmode\unskip\space\fi MR }
\providecommand{\MRhref}[2]{%
  \href{http://www.ams.org/mathscinet-getitem?mr=#1}{#2}
}
\providecommand{\href}[2]{#2}

\end{document}